\documentclass{amsart}
\usepackage{amsmath} 
\usepackage{amssymb}
\usepackage{mathtools}
\usepackage{centernot}
\usepackage{stmaryrd}
\usepackage{esvect}
\usepackage{amsthm}
\usepackage{upgreek}
\usepackage{comment}
\usepackage{amsfonts}
\usepackage{physics}
\usepackage{tikz-cd}
\usepackage[utf8]{inputenc}
\usepackage[english]{babel}
\usepackage{graphicx}
\usepackage[margin=1in]{geometry}
\usepackage{mathtools}

\usepackage{hyperref}
\usepackage{microtype}

\graphicspath{{images/}}

\renewcommand{\part}[1]{\noindent\textbf{Part #1)}}

\newcommand{\ba}{\begin{align*}}

\newcommand{\ea}{\end{align*}}

\newcommand{\Int}[1]{%
  {\kern0pt#1}^{\mathrm{o}}%
}

\newcommand{\bp}{\begin{pmatrix}}
\newcommand{\ep}{\end{pmatrix}}

\newcommand{\interior}[1]{%
  {\kern0pt#1}^{\mathrm{o}}%
}

\newcommand{\proj}{\pi} % no idea if this is good notation

\usepackage{xargs}                      % Use more than one optional parameter in a new commands
\usepackage{xcolor}
\usepackage[colorinlistoftodos,prependcaption,textsize=tiny]{todonotes}
\newtheorem{theorem}{Theorem}[section]
\newtheorem{corollary}{Corollary}[section]
\newtheorem{lemma}[theorem]{Lemma}
\newtheorem{claim}[theorem]{Claim}
\theoremstyle{definition}
\newtheorem{definition}{Definition}[section]

\title{Generalized Augmented cellular Alternating Links in Thickened Surfaces are Hyperbolic}
\author[C. Adams]{Colin Adams}
\address{Department of Mathematics and Statistics, Williams College, USA}
\email{cadams@williams.edu}
\author[M. Capovilla-Searle]{Michele Capovilla-Searle}
\address{Department of Mathematics, 14 MacLean Hall, Iowa City, Iowa 52242-1419}
\email{michele-capovilla-searle@uiowa.edu}
\author[D. Li]{Darin Li}
\address{4592 Terra Pl., San Jose, CA95121}
\email{darinli@gmail.com}
\author[Q. Li]{Qiao Li}
\address{Department of Mathematics, University of California, Berkeley, 970 Evans Hall, Berkeley, CA 94720-3840}
\email{lilyli@berkeley.edu}
\author[J. McErlean]{Jacob McErlean}
\address{120 Science Drive, 117 Physics Building, Campus Box 90320, Durham, NC 27708-0320}
\email{jacob.mcerlean@duke.edu}
\author[A. Simons]{Alexander Simons}
\address{6903 Rosemont Drive, McLean VA 22101}
\email{ads4@williams.edu}
\author[N. Stewart]{Natalie Stewart}
\address{Mathematics Department, Harvard University, 1 Oxford St, Cambridge, MA 02138}
\email{nataliestewart@math.harvard.edu}
\author[X. Wang]{Xiwen Wang}
\address{Room 501, No. 3, Lane 555, Guangyan Road, Shanghai, China 200072}
\email{xiwenw2@gmail.com}

\begin{document}  

\begin{abstract} Menasco  proved that nontrivial links in the 3-sphere with connected prime alternating non-2-braid projections are hyperbolic. This was further extended to augmented alternating links wherein non-isotopic trivial components bounding disks punctured twice by the alternating link were added. Lackenby proved that the first and second collections of links together form a closed subset of the set of all finite volume hyperbolic 3-manifolds in the geometric topology. Adams showed hyperbolicity for generalized augmented alternating links, which  include additional trivial components that bound n-punctured disks for $n \geq 2$. Here we prove that generalized augmented cellular alternating links in I-bundles over closed surfaces are also hyperbolic and that in $S \times I$, the cellular alternating links and the augmented cellular alternating together form a closed subset of finite volume hyperbolic 3-manifolds in the geometric topology. Explicit examples of additional links in $S \times I$ to which these results apply are included.
\end{abstract}

\maketitle

%TO-DO:
%\begin{itemize}
 %   \item Being careful about $M$ or $S \times I$.
 %   \item Words about $Q$ or $\mathring{N}(Q)$.
 %   \item I can redraw figures any time.
  %  \item I should be careful about talking about links v. their projection in the plane.
%\end{itemize}

\section{Introduction}

%{\color{magenta} Perhaps it's late to ask this, and this might be a bit general of a question, but do we not need an abstract?}

%{\color{blue} Yes, good point. I have added one.}

In a fundamental paper \cite{Menasco}, Menasco proved that a link in the 3-sphere with a reduced alternating projection that is connected, obviously prime, and not a 2-braid link projection has hyperbolic complement. 
Additionally, in \cite{Adams1}, it was proved that one can augment this link with a collection of nonisotopic trivial components that bound disks perpendicular to the projection plane such that the disks are punctured twice by the link and such that the new components puncture the projection plane in two points that are in non-adjacent complementary regions of the projection plane, and the resulting link is also hyperbolic.\\

These augmented alternating links have been investigated in a number of settings (e.g. \cite{Purcell} for background). In particular, it was shown that the the union of augmented alternating links with alternating links forms a closed subset of the set of all complete finite volume hyperbolic 3-manifolds in the geometric topology \cite{Lackenby}. The hyperbolicity results were further extended to generalized alternating links that allow the new components to bound disks with more than two punctures in \cite{Adams2}.\\

In \cite{SMALL2017} and \cite{HP}, results analogous to Menasco's results in the 3-sphere \cite{Menasco} were proved for links living in thickened surfaces, where now we project to the intermediate surface. There are various reasons to be interested in links living in thickened surfaces, but a primary one is that this is the appropriate ambient space for understanding virtual knots. See \cite{KK}, \cite{CKS} for the realization of virtual knots as knots in thickened surfaces and \cite{virtualknotssummer}  for hyperbolicity of virtual knots when realized in thickened surfaces. See also \cite{Turaevsummer} for knots living in thickened Turaev surfaces.\\

We extend results concerning generalized augmented alternating links \cite{Adams1,Adams2} to thickened surfaces via the following definitions.
Throughout this paper, we assume we are in the piecewise-linear category. 

%{\color{magenta} Maybe give a sentence introducing the concept of a projection onto $S \times \{\frac{1}{2}\}$, and henceforth refer to this as the \emph{projection surface}.}
\begin{definition}
Let $S$ be a closed surface and let $N$ be an I-bundle over $S$. Let $\proj: N \rightarrow S$ be the projection map that collapses $I$ over each point on $S$. %FEEL FREE TO CHANGE THIS IN THE PREAMBLE
A link $L \subset N$ has a  \emph{cellular alternating projection} if it has a projection to $S$ that is alternating and such that the complementary regions are all topological disks. (In previous papers, this was called a fully alternating projection, however this causes confusion with fully augmented alternating links, hence the change in terminology here).
%\textcolor{orange}{Where do we say something about avoiding 2-braids in the sphere and projective plane cases? And unlike other places, we are not using twisted I bundles over nonorientable surfaces. What about that?}{\color{red} For this definition it should work in all categories. This is not the theorem yet, it is just a definition at this point. I have changed to be as broad as possible.}
\end{definition}

%Note that in previous papers, this has been defined as cellular alternating, however, since the concept of ``fully augmented'' occurs in this paper, we are going with the other terminology for this. 

\begin{definition}
A link projection is \emph{reduced} if there is no disk in the projection surface intersecting the projection as in Figure \ref{reduced}.
\end{definition}

\begin{figure}[h]
    \centering
    \includegraphics[scale=.5]{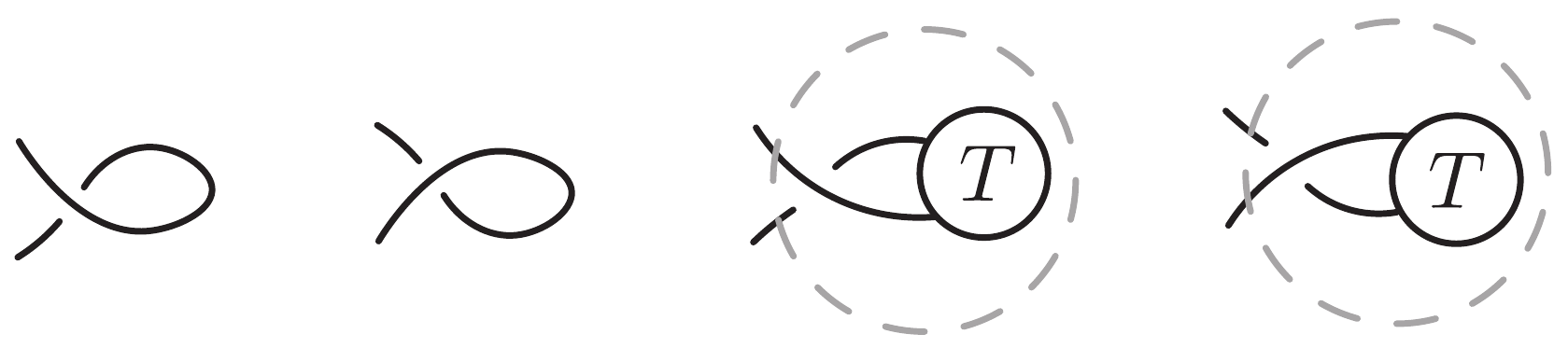}
    \caption{Crossings that can be reduced. Here, $T$ refers to the portion of the link projection contained in a disk in the projection surface whose boundary intersects the link exactly twice.}
    \label{reduced}
\end{figure}

In \cite{SMALL2017}, the following two theorems were proved.

\begin{theorem}
A link $L$ in an I-bundle $N$ over $S$ that has a reduced cellular alternating projection is \emph{prime} if its projection $P$ onto $S$ is \emph{obviously prime}, i.e., there do not exist circles in $S$ crossing $P$ transversely twice and bounding a disk in $S \times \{\frac{1}{2}\}$ that contains at least one crossing of $L$.
\end{theorem}

In the literature,  ``obviously prime'' is sometimes called ``weakly prime''.

\begin{theorem}
If $T$ is a torus, a link in $T \times (0,1)$  that has a reduced prime cellular alternating projection has hyperbolic complement. Similarly, if $S$ is a closed orientable surface of genus at least two, then a link in $S \times I$ that has a reduced prime cellular alternating projection has a hyperbolic complement such that $S \times \partial I$ is totally geodesic.
\end{theorem}

\begin{comment}
\begin{theorem}
A link that has a reduced prime cellular alternating projection in  $T \times (0,1)$ when the surface is a torus and in $S \times I$ when the surface has genus at least two has a hyperbolic complement, and when $S$ has genus at least two, the boundary surfaces can be realized as totally geodesic surfaces.
\end{theorem}
\end{comment}
%{\color{magenta} I've reworded this theorem a bit.}

These results were appropriately extended to non-orientable surfaces and twisted $I$-bundles over surfaces. Similar results were proved in a more general setting in \cite{HP}.\\

In this paper, we extend the category of generalized augmented cellular alternating links to such links in any $I$-bundle over any closed surface that is either orientable or non-orientable.

\begin{definition}
A {\it generalized augmented cellular alternating link} $Q$ in an I-bundle $N$ over a closed surface $S$ is obtained by starting with a link $L$ in $N$ that has a projection to $S$ that is reduced, prime and cellular alternating. Then we augment $L$ with trivial components perpendicular to the projection surface that bound disks punctured two or more times by $L$ and such that the new components puncture the projection surface in two nonadjacent complementary regions. Furthermore, if a pair of the trivial components intersect the projection surface in the same pair of complementary regions, they are not isotopic. We also assume the punctured disks are disjoint.
%{\color{magenta} Should this also assert that the augmenting components are nonisotopic, or that they intersect different complementary regions from each other?}{\color{blue}Fixed.}
\end{definition}

Fix an $I$-bundle $N$ over a closed surface $S$.
Define $M$ as follows. If $N$ has spherical boundaries, cap them off with balls. If $N$ has genus one boundaries, shave the boundaries off. Otherwise, $M = N$. 

%{\color{magenta} Don't we have to exclude $I$-bundles over the Klein bottle and $\mathbb{R}P^2$?}
%{\color{red}I don't see why, but we do have to be careful to exclude 2-braids for projective plane. }

\begin{theorem}\label{general theorem}%{\color{purple} As mentioned, have to exclude 2-braids in sphere and in projective plane.}
Let $Q$ be a generalized augmented cellular alternating link in $N$, an $I$-bundle over a closed surface $S$, excluding $S \cross I$ when $S$ is a projective plane. Then $M \setminus Q$ has a complete finite volume hyperbolic metric such that the boundary of the $I$-bundle is totally geodesic, unless one of the following occurs:
\begin{enumerate} 
\item $S$ is a sphere and $L$ is a 2-braid knot or link.
\item $S$ is a projective plane and $L$ is a 2-braid link.
\item $S$ is a projective plane and there exists a simple closed curve on $S$ that intersects the projection once.
\end{enumerate}

If $\chi(S) < 0$, then $\partial M$ can be realized as totally geodesic surfaces in the hyperbolic metric on $M \setminus Q$.
\end{theorem}

%{\color{red}This all needs to be reworded as we are really talking about $M$, not the I-bundle from whence it comes.}

%This result has been previously proved in a special case. A projection of a knot or link is twist-reduced if flypes have been applied to minimize the number of twist sequences, which are sequences of bigons attached end-to-end. Such a projection is said to be fully augmented if trivial components have been added to the link such that each encircles the waist of a different twist sequence, and all twist sequences are so encircled. It was proved in \cite{} that a fully augmented link in $S^3$ is hyperbolic. This also follows from \cite{Adams}, since once a projection is fully augmented, one can twist full twists on the augmented components to make the original projection alternating. However, in the former, geometric rather than topological techniques were applied to prove it.In \cite{Kwon}, these hyperbolicity results were extended to fully augmented links in a thickened torus. 

To prove Theorem \ref{general theorem}, we will apply Thurston's hyperbolicity criterion for 3-manifolds, which concerns essential surfaces.

\begin{definition} 
A properly embedded surface $S$ in a compact 3-manifold $M$ is \emph{incompressible} if either $S$ is a sphere that does not bound a ball in $M$ or $S$ not a sphere and there does not exist a disk $D$ in $M$ intersecting $S$ in $\partial D$ such that $\partial D$ is a nontrivial curve in $S$. The surface $S$ is \emph{boundary-incompressible} if there does not exist a disk $D$ with boundary consisting of two arcs $\alpha$ and $\beta$ such that $D \cap S = \alpha$, $D \cap \partial M = \beta$ and $\alpha$ does not cut a disk from $S$.
The surface is \emph{boundary-parallel} if it is isotopic into the boundary of the manifold relative to its boundary.
A properly embedded surface in a compact manifold is \emph{essential} if it is incompressible and not boundary-parallel. 
\end{definition}

%{\color{magenta} I thought we needed incompressibility and $\partial$-incompressibility?
%This seems like under the right circumstances, nonexistence of certain essential surfaces can be proved using only these conditions, but I feel weird about defining essential surfaces this way in general.\\
%}
%{\color{blue} boundary-compressibility only occurs for annuli, and you can prove that if an annulus boundary-compresses, and you know all spheres bound balls, then it must be boundary parallel.}
%{\color{magenta}
%$Also, this seems like a good place to state Thurston's criterion for hyperbolicity.\\ {\color{blue} Yes.}
%}
Thuston proved that if a compact orientable manifold with non-spherical boundaries contains no essential disks, spheres, tori or annuli, it is hyperbolic. Further any boundaries of genus at least two can be taken to be totally geodesic. (See \cite{Thurston}) It is this theorem that we will apply to prove Theorem \ref{general theorem}.\\ 

Note that a properly embedded disk $D$ is always boundary-incompressible, since $\alpha$ will always cut a disk from $D$.  Further, in the case of a properly embedded annulus in a manifold with no essential spheres, boundary-compressibility implies the annulus is boundary-parallel. Hence, once we prove there are no essential spheres, we will not need to consider boundary-incompressibility. \\

Theorem \ref{orientable theorem} in Section \ref{orientable section} discusses the case of Theorem \ref{general theorem} for thickened orientable surfaces, and then,  in Section \ref{non-orientable section}, we extend to non-orientable surfaces and to twisted I-bundles.\\

In Section \ref{rubber band section}, we utilize the result to prove hyperbolicity for rubber band links, a natural class of non-alternating links in thickened surfaces. In Section \ref{applications section}, using volume bounds from \cite{HP}, we extend results of Lackenby \cite{Lackenby} to show that in the thickened surface $S\cross I$, the union of all hyperbolic cellular alternating links and augmented cellular alternating links form a closed subset of the finite volume hyperbolic 3-manifolds with all boundary totally geodesic in the geometric topology. Further, we apply previous theorems from \cite{Lackenby}, \cite{KP}, and \cite{Kwon} to obtain upper and lower bounds on the volumes of rubber band links.\\
%{\color{purple} We are missing the case of a twisted I-bundle over an orientable surface.}

%{\color{blue} Give Kwon credit for torus case. Never mind. I realize she does just the fully augmented case for the torus, which we do not get into here.}

%Related work in the case of $T \cross (0,1)$ when the link is fully augmented (a different use of the word ``cellular'') appears in \ref{Kwon}.
We will jump between the complement of a link $L$ in a manifold, denoted $M \setminus L$ and the complement of the interior of a regular neighborhood of the link, denoted $M \setminus \mathring{N}(L)$ when we need to talk about properly embedded surfaces with boundary. 
\section{Generalized Augmented cellular Alternating Links in Thickened Orientable Surfaces}\label{orientable section}

\begin{definition}\label{main definition}
Let $L$ be a reduced prime cellular alternating link in $S \times I$, where $S$ is a closed orientable surface. Define $M$ as follows. When $S$ is a sphere, we cap off the surfaces $S \times \{0\}$ and $S \times \{1\}$ in $S \times I$ by balls to obtain $M = S^3$. When $S$ is a torus, we define $M = S \times (0,1)$. Otherwise we define $M = S \times I$. 
Considering the projection of $\proj(L) \subset S \times \{\frac{1}{2}\}$ as a 4-regular graph, the graph cuts $S \times \{\frac{1}{2}\}$ into complementary regions, all of which are topological disks.\\

Let $J_1, ..., J_n$ be $n$ loops in $S \times I \setminus L$ which are trivial in $S \times I$ such that

\begin{enumerate}
    \item each $J_i$ intersects $S \times \left\{\frac{1}{2}\right\}$ in exactly two points, one in each of two non-adjacent complementary regions,
    \item if two $J_i$ intersect $S \times \left\{\frac{1}{2}\right\}$ in the same pair of regions,  they are not isotopic in $S \times I \setminus L$. 
    %{\color{magenta} Isn't it unnecessary to make this conditional on having the same intersection regions? Isotopic loops have the same intersection regions.Can have two loops share same regions but be nonisotopic. One wraps one way around surface, other wraps a different way.}
    \item each $J_i$ bounds a disk $E_i$ perpendicular to $S \times \left\{\frac{1}{2}\right\}$ in $S \times I$ where $E_i \cap E_j = \varnothing$ for $i \neq j$.
\end{enumerate}

Then we call the link $Q = L \cup \left(\bigcup\limits_{i=1}^n J_i\right)$ a \textit{generalized augmented cellular alternating link} in $M$, and call each $J_i$ an \textit{augmenting component}.
\end{definition}

Figure \ref{augmented example} gives an example.

\begin{figure}[h]
    \centering
    \includegraphics[scale=.5]{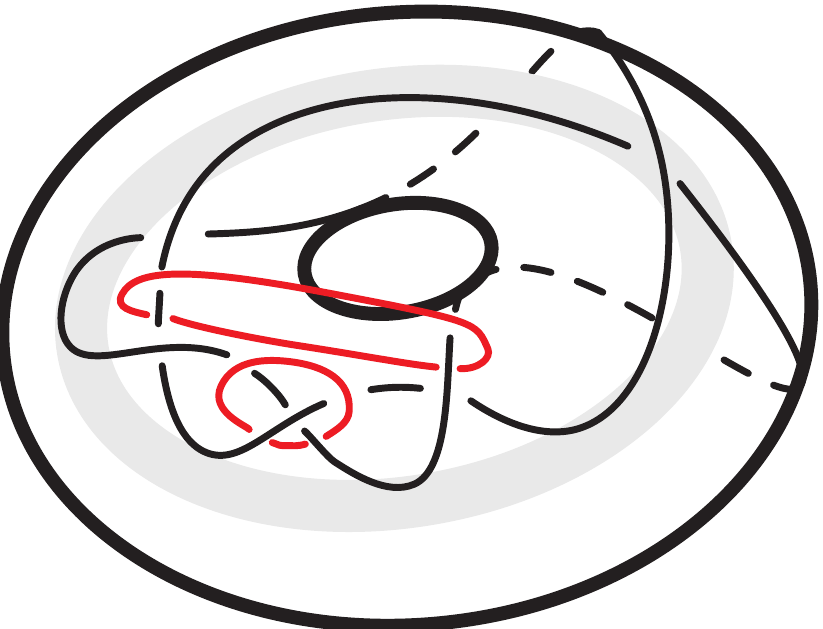}
    \caption{A generalized augmented cellular alternating link.}
    \label{augmented example}
\end{figure}

%For notational convenience, when we refer to $S \cross I$, we allow $I$ to correspond to the interval (0,1) when $S$ is a torus and $[0,1]$ when $S$ is higher genus.  

\begin{theorem}\label{orientable theorem}
Let $Q$ be a generalized augmented cellular alternating link in $M$ as obtained from $S \times I$ and $L$ as above,  where $S$ is a closed orientable surface, other than $L$ a 2-braid with $S = S^2$. Then $M \setminus Q$ has a complete hyperbolic metric with finite hyperbolic volume. When $S$ has genus at least 2, both of the surfaces $S \times \{0\}$ and $S \times \{1\}$ are totally geodesic in the metric.
\end{theorem}

In order to prove this theorem, we will prove a sequence of lemmas. We assume throughout that the genus of $S$ is at least one, since the theorem is known when $S$ is a sphere and therefore $M = S^3$ (see \cite{Adams2}). As mentioned previously, we rely on Thurston's theorem which says that in the case of  a link in a compact 3-manifold, the complement is hyperbolic if and only if it contains no essential disks, spheres, tori, or annuli \cite{Thurston}. \\

Note that we do not need to consider essential disks separately, since $\partial(S \times I)$ is incompressible in $S \times I$, so there are no essential disks with boundary on $\partial(S \times I)$.  Further, by hyperbolicity of the complement of $L$, there are no essential disks with boundary on $\partial N(L)$, where $N$ is a regular neighborhood of $L$. So if there is an essential disk $D$, it must have nontrivial boundary on  $N(J)$. Then, $\partial N(D \cup J)$ is an essential sphere, which we will eliminate.

\begin{lemma}\label{lift lemma}
Suppose a link $L$ in $S \cross I$ has a projection $P$ that is reduced, cellular alternating, and obviously prime in $M$. Then $\tilde{P}$, the lift of $P$ in the universal cover of $M$, is also connected, reduced, alternating, and obviously prime.
\end{lemma}

\begin{proof} 
The universal cover of $M$ is the thickened Euclidean or hyperbolic plane, which we consider as $\mathbb{R}^2 \times (0,1)$ or $\mathbb{R}^2 \times I$. 
Fix $L$ so that its projection is $P$. Since $P$ is cellular alternating, its complementary regions are disks. Therefore it is connected and so is $\proj(\tilde{L}) = \tilde{P}$. Also since $P$ is cellular alternating, $\tilde{P}$ is cellular alternating.\\

Suppose $\tilde{P}$ is not reduced or not prime. Then there exists a disk $D \subset \mathbb{R}^2$ such that $\partial D$ intersects $\tilde{P}$ exactly twice and $D$ contains crossings of $\tilde{P}$.  On $S$, then, $\partial D$ corresponds to a homotopically trivial curve bounding an immersed disk $D' \subset S$. We can isotope $D$ slightly on $\mathbb{R}^2$ so that $\partial D'$ intersects itself tranversely if at all. Since $\partial D$ intersects $\tilde{P}$ exactly twice, $\partial D'$ intersects $P$ exactly twice as well. The boundary of $D'$ cuts the immersed disk $D'$ into subdisks, the boundary of only one of which intersects the projection and does so twice. The fact $D$ contained crossings implies this subdisk of $D'$ contains crossings, a contradiction to $P$ being reduced and obviously prime.
\end{proof}

\noindent{\bf Construction of the link $Q'$:}

\medskip

%Given an essential surface in a generalized augmented cellular alternating link in a thickened surface, we construct an associated link in $S^3$ which will allow us to prove that the essential surface is not there. 

%{\color{magenta}
%I find terminology here a bit confusing;
%is a ``lift'' of $F$ one of the smallest (properly embedded) surfaces in the preimage of $F$ under the cover s.t. the surface contains at least one preimage for each point of $F$?
%This seems to be what is meant, but some terminology such that $F$ ``lifts to a compact surface'' is weird if that's the case.\\
%}{\color{blue} Is the problem that it doesn't say a collection of compact surfaces? I have changed that.}

Given an essential surface $F$ that might appear in $M \setminus Q$, and that lifts to a collection of compact surfaces in $\mathbb{R}^2 \cross I$, we construct a link $Q'$ in $S^3$ that will allow us to prove the nonexistence of the essential surface.
The universal cover of $S \times I$ is the thickened Euclidean or hyperbolic plane. Since each lift of $F$ is compact, there are an infinite number of lifts of $F$. Let $\tilde{F}$ be one such surface in the thickened plane. Consider $\tilde{Q}$, the lift of $Q$, as a link in $\mathbb{R}^2 \times I$. \\

Let $R$ be a particular fundamental domain in $\mathbb{R}^2$ that projects to $S$. Then we can choose $R$ to be a polygon with four edges when $S$ is a torus, and more when $S$ has a higher genus. We know $\partial R$ must intersect $\tilde{P}$ in at least eight points as follows. No side of $R$ can have zero intersections with $\tilde{P}$, as if it did, not all the complementary regions of $P$ on $S$ would be disks. In addition, each side of $R$ must intersect $\tilde{P}$ in an even number of points, or $P$ could not be alternating.  Since $R$ has at least four sides, $\partial R$ has at least eight points of intersection with $\tilde{P}$. \\

Now consider the projection of $\tilde{F}$ onto $\mathbb{R}^2$. Since $\tilde{F}$ is compact, we can take the union of a finite number of copies of the fundamental domain for the covering map of $\mathbb{R}^2$ onto $S$ that is a topological disk $D \subset \mathbb{R}^2$ and that contains all of the projection of $\tilde{F}$. We choose a particular $R$ to be a fundamental domain in the interior of $D$ that intersects the projection of $\tilde{F}$.\\

%Let $R$ be a particular fundamental domain in $\mathbb{R}^2$, in the interior of $D$. Then $R$ a polygon with 4 edges when $S$ is a torus, and more when $S$ has a higher genus. Therefore, there are no intersections between $\partial D$ and the lifts of augmenting components of $Q$ that cannot be avoided via isotopy. Furthermore, we know $\partial R$ must intersect $L$ in at least 8 points: no side of $R$ can have 0 intersections with $L$, or not all complementary regions of $L$ on $S$ are disks. In addition, each side of $R$ must intersect $L$ in an even number of points, or $L$ would not be alternating.  Since $R$ has at least 4 sides, $\partial R$ has at least 8 points of intersection with $L$.\\

For each augmenting component in $M$, choose a connected component in its lift to $\mathbb{R}^2 \times I$  with projection that intersects $R$. Call these components $W$ and let the projection of these components to $\mathbb{R}^2 \times \{1/2\}$ be called $W'$.  Then because $\tilde{P}$ is connected, we can choose a larger disk $D'$ containing $D$ and made up of copies of the fundamental domain such that $W'$ intersects $\tilde{P} \cap D'$ in a single connected component denoted $E$. By construction, there will be such a component which, in particular, contains all of $R \cap \tilde{P}$. \\

%By Lemma \ref{lift lemma}, since $P$ is cellular alternating and therefore connected, $\tilde{P}$ is also connected. Now consider the portion of $\tilde{P}$ that is contained in $D$, ignoring for the moment any augmenting components that might be contained in $D$. This need not be connected, so for any components of $\tilde{P} \cap D$ that intersect the $\partial D$ in only two points discard them. 

 We call $E  \cap \partial D'$ the endpoints of $E$. Since $\partial R$ intersects $\tilde{P}$ at least eight times, there will be at least eight endpoints of $E$ on $D'$, and the total number must be even. So we can connect them in pairs to obtain a new link projection $\overline{E}$ in the plane.\\ 
 
 In fact, we can always close off $E$ such that $\overline{E}$ is alternating. We do this by first numbering the endpoints $1, \ldots, 2n$ clockwise. Then we add a simple arc outside $D'$ from endpoint 1 to endpoint $n+1$ clockwise, then a simple arc from endpoint 2 to endpoint $n+2$ counterclockwise, etc. until all endpoints are connected. Call each of these arcs an \emph{embroidery arc}, and the collection of the $n$ embroidery arcs the \emph{embroidery of $E$}. See Figure \ref{closure}. This procedure creates $n-1$ crossings on each embroidery arc, and since endpoints must alternate in terms of whether the strand leaving the endpoint is an understrand or an overstrand at the first crossing it encounters in $D'$, by considering the cases of $n$ even or $n$ odd, we see that we can always assign the new crossings so that $\overline{E}$ is alternating, without changing any existing crossing of $E$.\\

\begin{figure}[h]
    \centering
    \includegraphics[scale=.4]{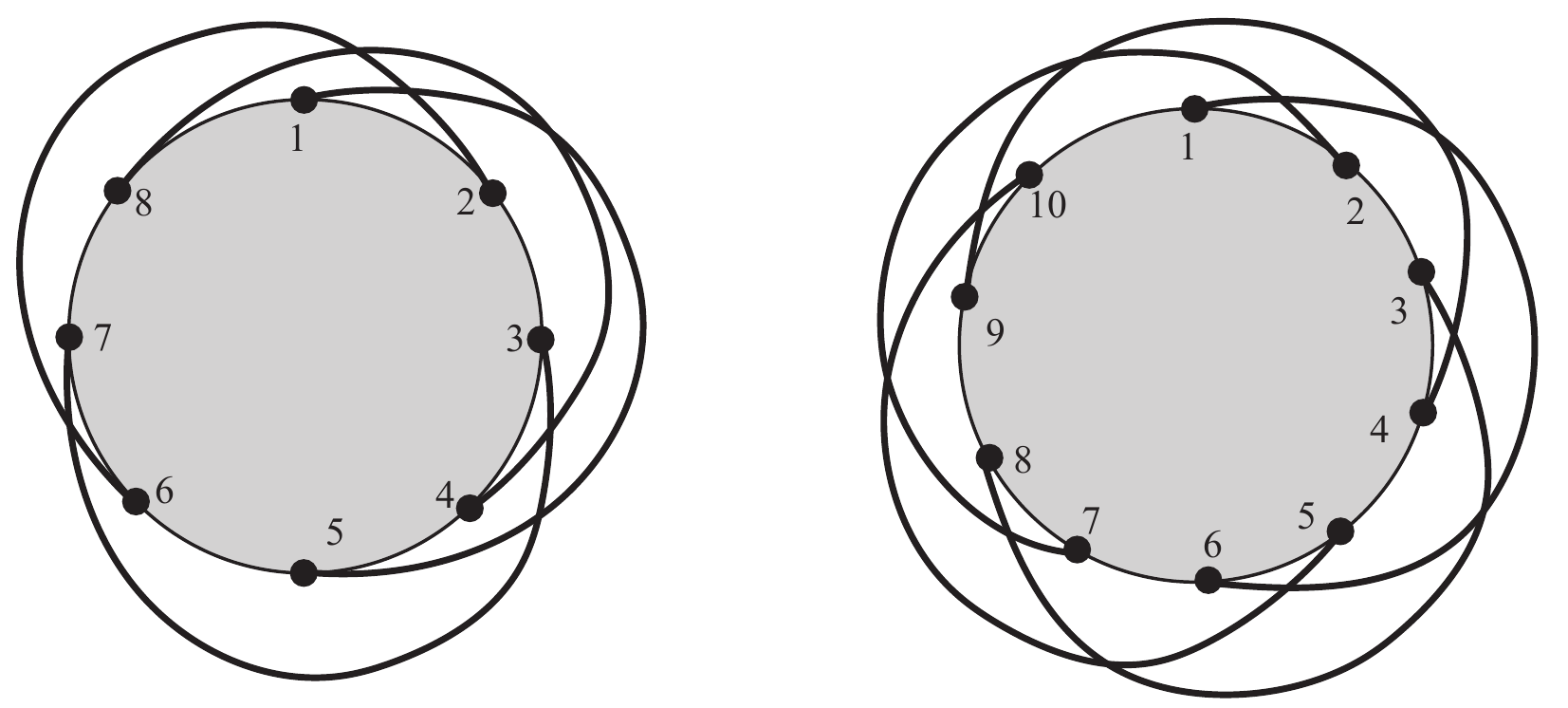}
    \caption{Closing off $E$ with embroidery.}
    \label{closure}
\end{figure}

\begin{claim} The projection $\overline{E}$ is alternating, connected, reduced and obviously prime. It corresponds to neither a 2-braid nor a trivial knot. 
\end{claim}

\begin{proof} With our construction, we have ensured that $E$ is connected, and so $\overline{E}$ is also connected. Next, we show that $\overline{E}$ is reduced. By Lemma \ref{lift lemma}, there cannot be reducible crossings as depicted in Figure \ref{reduced} entirely contained in $D'$. Adding the embroidery arcs cannot create Type I crossings, since no adjacent endpoints are connected. If adding the embroidery arcs created a new crossing on one of these arcs that can be removed by a flype, $E$ would not be connected. If adding the embroidery arcs turned an existing crossing in $E$ into one that can be removed by a flype, there must exist a disk $D''$ containing a portion $T$ of $\overline{E}$, whose boundary intersects $\partial D'$ exactly twice and intersects $\overline{E}$ exactly twice near the crossing. See Figure \ref{no flype}. Since $T$ contains an embroidery arc, such a disk $D''$ cannot exist given how we choose to close off $E$.\\

\begin{figure}[h]
    \centering
    \includegraphics[scale=.5]{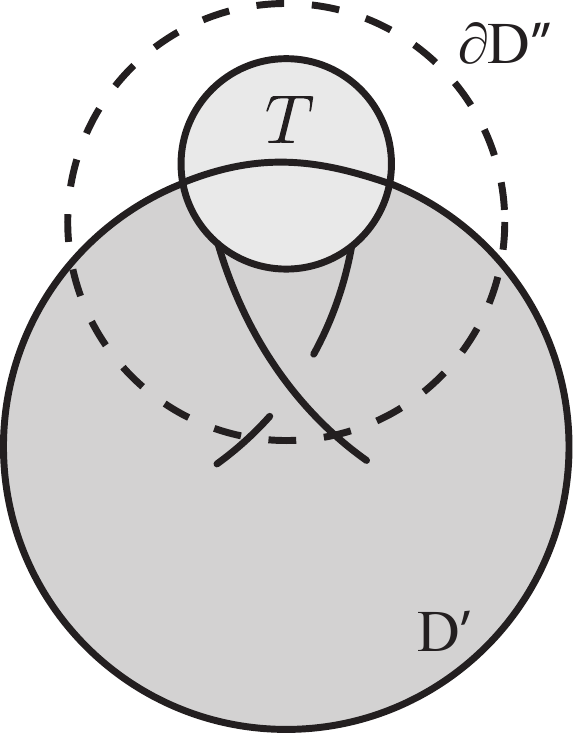}
    \caption{Suppose the addition of embroidery arcs created a flype crossing in $D'$.}
    \label{no flype}
\end{figure}

We now show that $\overline{E}$ is obviously prime. Suppose there exists a disk $D'' \subset \mathbb{R}^2$ containing a crossing of $\overline{E}$ such that $\partial D''$ crosses the projection $\overline{E}$ transversely twice. By Lemma \ref{lift lemma}, both intersections on $\partial D''$ cannot be with $E$. If one intersection on $\partial D''$ is with an embroidery arc and another is with $E$, given our construction of closing off $E$, we can always isotope $\partial D''$ to be contained in $D'$, again contradicting the obvious primeness of the original projection $P$.\\ 

If both intersections on $\partial D''$ are with embroidery arcs, consider the complementary regions of the union of the embroidery arcs of $\overline{E}$ on the plane. See Figure \ref{flower}. Given our scheme for embroidering $E$ and the requirement that there be at least eight endpoints, and that $\partial D''$ only intersects $\overline{E}$ twice, $\partial D''$ can only  cross a sub-arc of an embroidery arc that is on the boundary or interior of the central region.  So consider only arcs of the embroidery arcs that are incident to the central region. They must have the pattern of a flower, perhaps with two petals sliced in half. 
%Then, $\partial D''$ must intersect $\partial D'$. If the intersection $D' \cap D'' = \varnothing$,{\color{blue} Need to fix this.} i.e., 
If no portion of $E$ is contained in $D''$, then the crossing in $D''$ must result from embroidery arcs. Then there must be endpoints of embroidery arcs inside $D''$, contradicting the assumption that $D'' \cap E = \varnothing$. Otherwise, if there exists a portion of $E$ contained in $D''$, $P'$ would not be connected. This is a contradiction to our construction, where we remove any non-connected piece in $D'$ before we embroider. Therefore, by Theorem 1 of \cite{Menasco}, $\overline{E}$ is prime.\\

\begin{figure}[h]
    \centering
    \includegraphics[scale=.5]{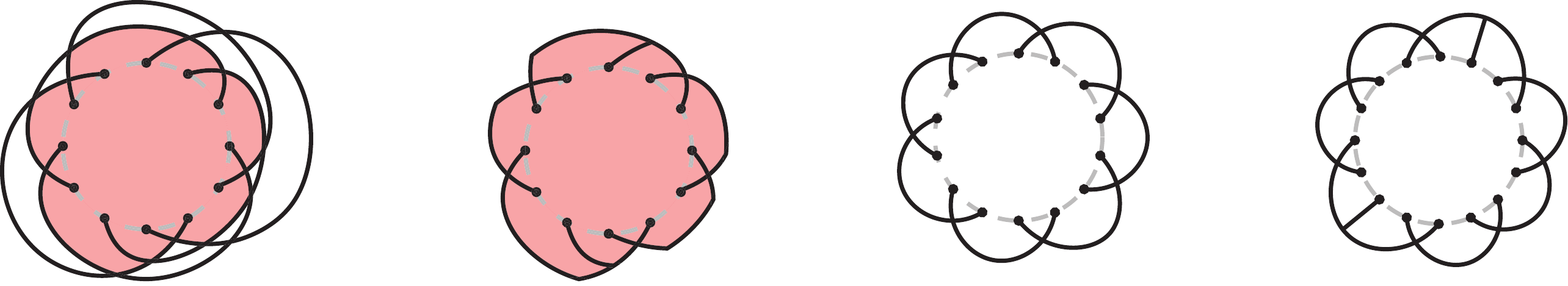}
    \caption{Complementary regions of embroidery arcs. The central region is shaded and always has the shape of a flower, perhaps with two bisected petals.}
    \label{flower}
\end{figure}

Note also that $\overline{E}$ is nontrivial. By Theorem 5.8 of \cite{Crowell}, there is no reduced alternating projection of the trivial knot.\\

Finally, we demonstrate that $\overline{E}$ is not a 2-braid. Suppose that it were. By the Flyping Theorem of \cite{MT1991} and \cite{MT1993}, the only reduced alternating projection of a 2-braid is the standard one. Then as in our consgruction, the disk  $D'$ appears in  the standard projection of a 2-braid such that all arcs outside $D'$ respect the rules of our embroidery construction. If $\partial D'$ does not intersect the bigon regions of a 2-braid, then $D'$ would either have no crossings or no endpoints, which is impossible. If $\partial D'$ intersects a bigon region, it cannot intersect the same side of the bigon region twice consecutively, for otherwise two consecutive endpoints would be connected by an embroidery arc, or there would exist a non-connected portion of $E$ in $D'$. So, $\partial D'$ must cross the two sides of a bigon region in immediate succession. Moreover, $\partial D'$ must then immediately cross the two sides of another bigon region. See Figure \ref{star}. Of these four consecutive endpoints on $\partial D$, the embroidery arcs either connect the first to the fourth and the second to the third, or the first to the third and the second to the fourth, neither of which occurs for  a valid embroidery pattern.\\

%, the first two have embroidery arcs leaving in the same direction, and the latter two also have embroidery arcs leaving in the same direction. However, this is impossible given our embroidery.\\

\begin{figure}[h]
    \centering
    \includegraphics[scale=.3]{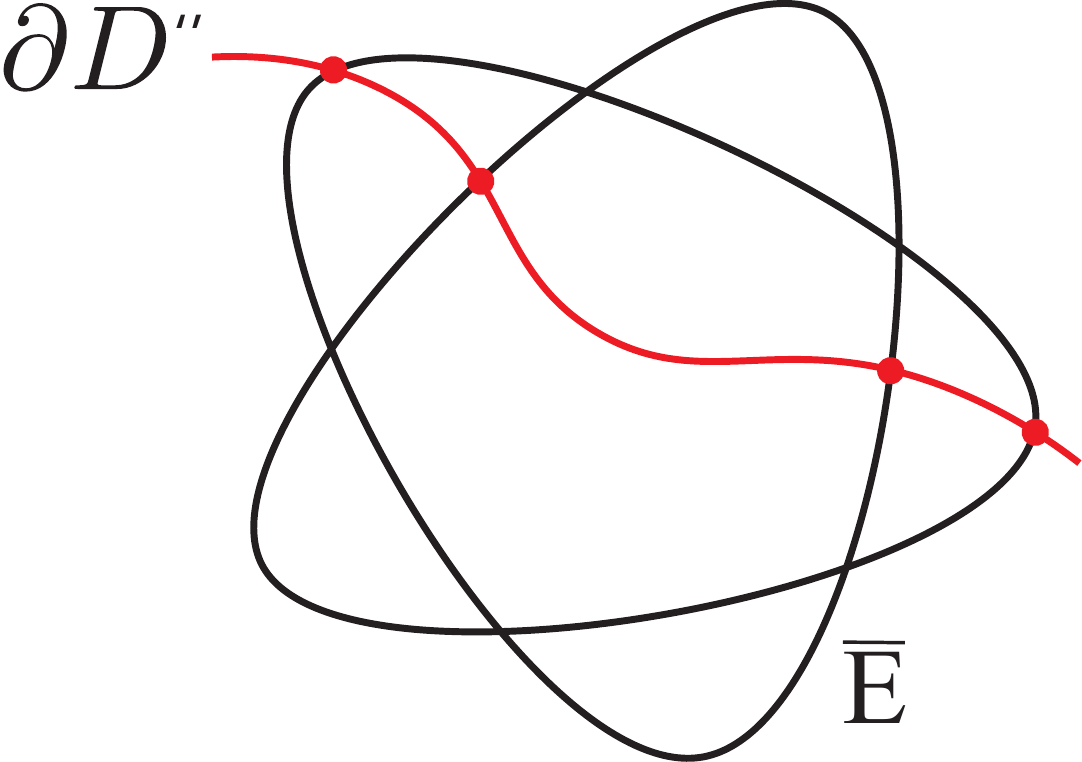}
    \caption{Suppose $\overline{L'}$ is a 2-braid.}
    \label{star}
\end{figure}

Therefore, $\overline{E}$ is a reduced connected prime alternating link projection that does not correspond to a trivial knot or a 2-braid.
\end{proof}

Adding to $\overline{E}$ the augmenting components $W$ with projections in $W'$, we obtain a generalized augmented alternating link $Q'$ in  a compact subset of $\mathbb{R}^2 \times I$. So we can take $\mathbb{R}^2 \subset S^2$, and cap off $S^2 \cross I$ with two balls to embed $Q'$ in $S^3$.   By \cite{Adams2}, it must have hyperbolic complement.

\begin{lemma}\label{spheres}
There are no essential spheres in $M \setminus Q$.
\end{lemma}

\begin{proof}
Suppose there exists an essential sphere $F$ in $M \setminus Q$. Note that since $L$ is a cellular alternating link in $S \times I$, we know by the hyperbolicity of $M \setminus L$ (\cite{SMALL2017}that no components of $L$ can be contained in the ball $B$ bounded by $F$. So it can only be augmenting components in $B$ that make $F$ essential. \\

From our construction, we know that there is a copy $\tilde{F}$ of $F$ in $S^3 \setminus Q'$ that bounds a ball $\tilde{B}$ in $S^2 \cross I$. Because $S^3 \setminus Q'$ is hyperbolic, $\tilde{F}$ must bound two balls in $S^3$ such that all components of $Q'$ must be in one of them. If $\tilde{B}$ is the ball containing no components of $Q'$, then it projects to $B$ in $M \setminus Q$, a contradiction to the fact $B$ contained components of $Q$. If on the other hand, $\tilde{B}$ is the ball containing all of $Q'$, then $B$ must contain all of $Q$. In particular, it contains all of $L$, which is a contradiction to the fact that $L$ is a prime reduced cellular alternating link on the surface and therefore hyperbolic.  
\end{proof}

%{\color{blue} CA: First need to prove that $Q$ is prime. Then can prove no essential tori.}

\begin{lemma}\label{primeQ} The link $Q$ is prime in $M$.
\end{lemma}

\begin{proof} Suppose not. Then there is a sphere $F$ in $M$ punctured twice by $Q$ that bounds a ball $B$ in $M$ such that $Q \cap B$ is not an unknotted arc. Let $\tilde{B}$ be a lift of $B$ to $\mathbb{R}^2 \cross I$ and let $\tilde{F} = \partial \tilde{B}$. Then as in the construction, we can associate the projection $\overline{E}$ to $\tilde{F}$. Adding in the augmenting components, we obtain a link $Q'$ in $S^3$ which is a generalized augmented alternating link and therefore hyperbolic. But it contains $\tilde{B}$, which is a ball intersecting $Q'$ in a portion that is not a trivial arc. This contradicts hyperbolicity of a generalized augmented alternating link in $S^3.$
\end{proof}

We next consider essential tori. The following lemma will prove useful.

\begin{lemma}\label{tori lemma}
    Suppose $S$ is a surface other than the torus, and suppose $T \subset S \times I$ is a torus.
    Then, $T$ is compressible in $S \times I$.
\end{lemma}

\begin{proof}
 The fundamental group of the torus is $\pi_1(T)=\mathbb{Z}\oplus \mathbb{Z}$, and the fundamental group of a thickened surface $S\cross I$ is $\pi_1(S\cross I)=\pi_1(S)$. Apart from the torus, no other surface's fundamental group contains $\mathbb{Z}\oplus\mathbb{Z}$ as a subgroup. 
 Therefore the induced fundamental group homomorphism must have a nontrivial kernel;
 by Dehn's Lemma (Corollary 3.2 in \cite{Hatcher}), the contraction of a nontrivial element of this kernel gives a disk in $S\cross I$ with boundary a nontrivial curve on $T$.
 Therefore the torus $T$ has an embedded compression disk in $S\cross I$.
\end{proof}

\begin{lemma}\label{tori}
There are no essential tori in $M \setminus \mathring{N}(Q)$.
\end{lemma}

\begin{proof}
Suppose for the sake of contradiction that $T$ is an essential torus in $S\cross I\setminus Q$. By Lemma \ref{tori lemma}, we can split our discussion into two cases:
\begin{enumerate}
    \item $S$ is a torus and $T$ is isotopic to $S \cross \{0\}$ in $S\cross I$,
    \item $T$ is compressible in $S\cross I$.
\end{enumerate}

In Case (1), suppose $S$ is a torus, and $T$ is parallel to a component of the boundary of $S\cross I$.
%, as in Figure \ref{torus boundary parallel}. 
Note that $T$ separates $S\cross I$ into two thickened tori. Then, $L$ must be contained entirely in one of the thickened tori, as $S\cross I \setminus L$ would otherwise have an essential torus $T$ and fail to be hyperbolic. If $L$ is in the thickened torus to one side of $T$, there must be augmenting components $J$ to the other side of $T$ in $S\cross I$ to prevent $T$ from still being boundary parallel in $S\cross I\setminus Q$. But since each augmenting component is trivial in $S \cross I$, it bounds a disk that misses $T$ and that therefore is not punctured by $L$. If we remove all but one of the augmenting components to this side of $T$, we see that the remaining component bounds a disk, the boundary of a neighborhood of which is a sphere separating this component from the rest of the link. This is a contradiction to the nonexistence of essential spheres, as shown in Lemma \ref{spheres} for the link consisting of $L$ and this one augmenting component.\\

%\begin{figure}[h]
 %   \centering
 %   \includegraphics[scale=.3]{torus boundary parallel.pdf}
 %   \caption{$T$ parallel to the boundary.}
 %   \label{torus boundary parallel}
%\end{figure}

In Case (2), suppose $T$ is compressible in $S\cross I$. Surgering along the compression disk yields a sphere, which must bound a ball to one side. So $T$ either bounds a solid torus or a knot exterior which itself is contained in a ball in $S \cross I$. Lift $S \cross I$ to its universal cover $\mathbb{R}^2 \times I$.\\

If the torus bounds a solid torus with core curve that is homotopically trivial  in  $S \times I$ or it bounds a knot exterior, then $T$ lifts to infinitely many tori, each copy of which is contained within a ball in $\mathbb{R}^2 \cross I$. See Figure \ref{lifting torus to balls}. (When the core curve is homotopically trivial but not trivial, the solid tori can be linked and the balls need not be disjoint.)   Let $\tilde{T}$ be a particular lift of $T$ and let $B$ be a ball in $\mathbb{R}^2 \cross I$ that contains it. Then $\tilde{T}$ separates a solid torus or a knot exterior from $\mathbb{R}^2 \cross I$.  As in our construction above, we obtain the complement of a new generalized augmented alternating link $Q'$ in the thickened plane. If we consider it in $S^3$, then since we know it is hyperbolic, $\tilde{T}$ must be boundary parallel or compress in the complement of $Q'$.\\

\begin{figure}[h]
    \centering
    \includegraphics[scale=.5]{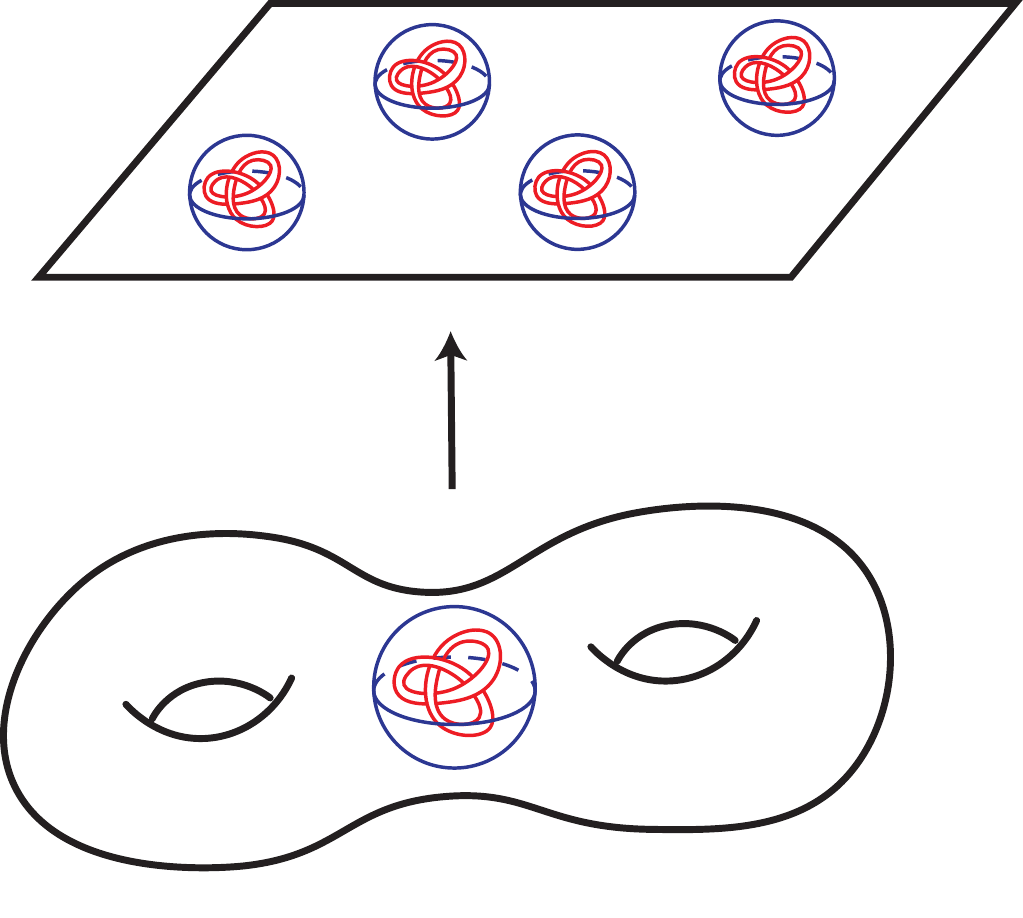}
    \caption{$T$ compressible in $S\cross I$ and contained in a ball.}
    \label{lifting torus to balls}
\end{figure}

First consider the case where $\tilde{T}$ is boundary parallel in $S^3 \setminus Q'$. So $\tilde{T}$ bounds a solid torus $\tilde{V}$ in $S^3$ with core curve a component $C$ of $Q'$, and $\tilde{T}$ separates that component from the rest of $Q'$. Note that a torus in $S \cross I$ cannot separate the two boundary components. Therefore, $V$ either contains the two balls we added to obtain $S^3$, or it contains neither. If it contains neither, then projecting $\mathbb{R}^2 \cross I$ back down to $S \cross I$ means that $T$ is boundary parallel in $S \cross I \setminus Q$. If $V$ does contain the two balls, then we can isotope a curve on $\tilde{T}$ to a meridian of $C$ to obtain an annulus. That annulus projects to an annulus in $M \setminus Q$ that contradicts the primeness we proved in Lemma \ref{primeQ}.\\

%First, consider the case where $\tilde{T}$ bounds a solid torus in a ball. If $\tilde{T}$ is boundary parallel in the complement of $Q'$, it can be boundary parallel to the solid torus side within the corresponding ball in $S^3$ or to the outside if $\tilde{T}$ is unknotted. In either case, $\tilde{T}$ separates one component of $Q'$ from the rest.  there is a single link component in $Q'$ which is the core curve of the solid torus. It must also be true that in $\mathbb{R}^2 \times I$, $\tilde{T}$ is still boundary parallel to the core curve of the solid torus it bounds. This further holds when we project down to $S\cross I\setminus Q$, which is a contradiction to our assumption the initial torus is essential.\\ 

If on the other hand, $\tilde{T}$ is compressible in the complement of $Q'$, then if $T$ bounds a knot exterior in $S \cross I$, the lift $\tilde{T}$ does also in $S^3$. So the compression must be to the outside, and therefore the knot exterior that $\tilde{T}$ bounds in $S^3 \setminus Q'$ is contained in a ball that either contains all of $Q'$ or none of it.  In either case, the compression disk projects to a compression disk for $T$ in $S \cross I \setminus Q$, a contradiction.\\

If $\tilde{T}$ is compressible in the complement of $Q'$, and $T$ bounds a solid torus $V$ in $S \cross I$, then $\tilde{T}$ bounds a solid torus $\tilde{V}$ in $S^3$ that is contained in $\mathbb{R}^2 \cross I$ considered as a subset of $S^3$. If the compression is to the inside of $\tilde{V}$, then the compression disk projects to a compression disk of $V$ in $S \cross I \setminus Q$, a contradiction to the essentiality of $T$.\\

If the compression is to the outside of $\tilde{V}$, then again, the sphere resulting from the compression must contain all or none of $Q'$. If it contains none of $Q'$, then $\tilde{V}$ contains none of $Q'$ and projects to $V$ containing none of $Q$. But then $T$ compresses through $V$, contradicting essentiality. 
However, it cannot contain all of $Q'$ since $\tilde{V}$ was contained in $D \cross I \subset \mathbb{R}^2 \cross I$, and the construction of $\overline{E}$ involved adding embroidery outside $D \cross I$. \\

If the core curve $\gamma$ of the solid torus $V$ bounded by $T$ is not a homotopically trivial knot in $S \times I$,  then the core curve lifts to a collection of infinite knots in the universal cover, and the solid torus lifts to regular neighborhoods of those infinite knots.  Each tube spans an infinite number of fundamental domains, and thus, we cannot use the previous construction of $\overline{E}$. Hence, we describe a second construction here.\\

The curve $\gamma$ is homotopic to a mutiple of a geodesic $\gamma'$ on $S$.  If $\gamma'$ is nonseparating, let $\alpha$ be a geodesic on $S$ that intersects $\gamma'$ once. If $\gamma'$ is separating, as can occur when the genus is greater than one,  choose $\alpha$ to be a geodesic that intersects it twice.\\

We can choose a fundamental domain for the projection of $\mathbb{R}^2$ to $S$ such that two nonadjacent edges lie in lifts of $\alpha $. Let $\gamma''$ be a choice of a lift of $\gamma'$ to the universal cover of $S$, which is either $\mathbb{E}^2$ or $\mathbb{H}^2$, but which we now denote $\mathbb{R}^2$. Note that there is a discrete subgroup of the isometries of the universal cover, 
such that the quotient of $\gamma''$ under the action of the group is $\gamma'$.
Choose $R$ to be a copy of the fundamental domain for the projection of $\mathbb{R}^2$ to $S$ that intersects $\gamma''$.\\

Let $\tilde{V}$ be lift of $V$ to $\mathbb{R}^2 \cross I$ such that its projection to $\mathbb{R}^2$ intersects $R$.
For each augmenting component in $M$, choose a lift to $\mathbb{R}^2 \times I$  with projection to $\mathbb{R}^2$ that intersects $R$. Call these components $W$ and let the projection of these components to $\mathbb{R}^2 \times \{1/2\}$ be called $W'$. \\ 

Because $\tilde{P}$ is connected, we can choose a disk $D$ in $\mathbb{R}^2$ such that the following hold:

\begin{enumerate}
\item $D$ is a  made up of a union of fundamental domains isometric to $R$, including $R$ itself.

\item Two disjoint edges on the boundary of $D$ lie in two lifts $\alpha_1$ and $\alpha_2$  of $\alpha$ that are identified by a covering translation in the subgroup that acts on $\gamma''$.

\item A connected subset of $\tilde{V}$ that projects to all of $V$ in $S \cross I$ has projection to $\mathbb{R}^2$ that is contained in $D$. 
\item $W'$ intersects $\tilde{P} \cap D$ in a single connected component, which we denote  $E$. 

\item The number of endpoints of $E$ to either of the two sides of $D$ that are not in lifts of $\alpha$ are both even and at least eight.
\end{enumerate}

Condition (1) is immediate.
Condition (2) can be satisfied because $\tilde{V}$ stays within a finite fixed distance of $\gamma''$.
Conditions (3),(4) and (5) can be satisfied by taking $D$ large enough.\\

We add  any necessary connected components of $\tilde{P} \cap D$ to $E$ to obtain $E'$ so the number of endpoints on the two sides of $D$ on $\alpha_1$ and $\alpha_2$ are the same. See Figure \ref{coverdomain}(a) and (b) for the Euclidean and hyperbolic cases, where the red arcs correspond to $W'$. \\

\begin{figure}[ht]
    \centering
    \includegraphics[scale=.5]{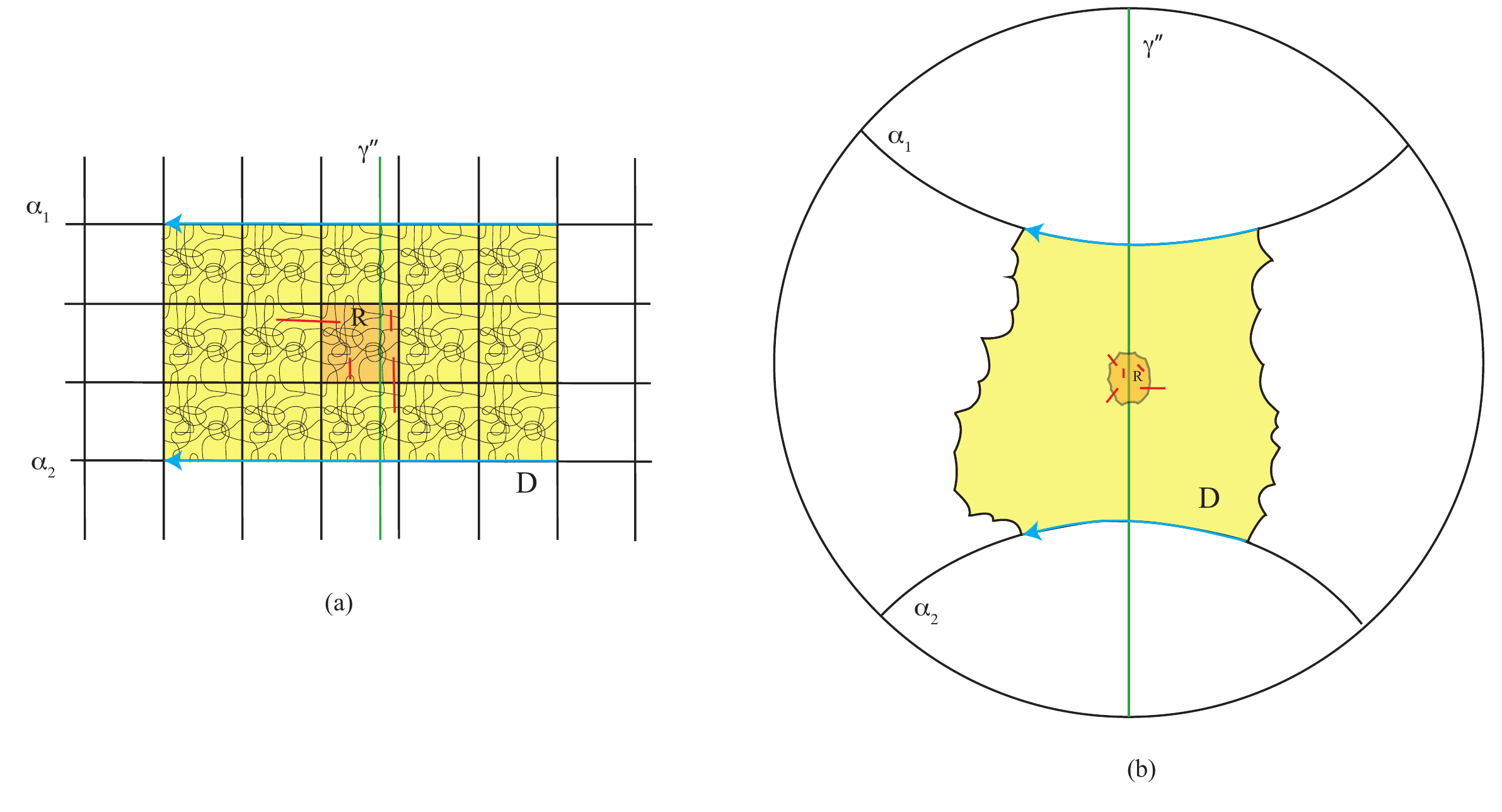}
    \caption{Construction of disk $D$ in $\mathbb{E}^2$ and $\mathbb{H}^2$.}
    \label{coverdomain}
\end{figure}

Note that because there is a covering translation of $\mathbb{R}^2$ that identifies the two boundary edges of $D$ on $\alpha_1$ and $\alpha_2$, denoted by blue arrows in the figure, we obtain from it an annulus. Identifying the corresponding endpoints of $E'$, we obtain a link projection $E''$ in an annulus $A$, with an even number of endpoints on each boundary.\\

Then, we separately embroider endpoints of $E''$ on the outer and inner boundaries of $A$ in the manner discussed in the proof of Lemma \ref{spheres}, as in Figure \ref{closing off annulus}. Call the result $\overline{E''}$. Let $L'$ be the link in $S^3$ that corresponds to the link with projection $\overline{E''}$, and let $Q'$ be $L'$ together with the components corresponding to the augmenting components $W$. Condition (3) implies there is a torus $T'$ corresponding to $T$ in the complement of $Q'$ in $S^3$.\\

\begin{figure}[ht]
    \centering
    \includegraphics[scale=.3]{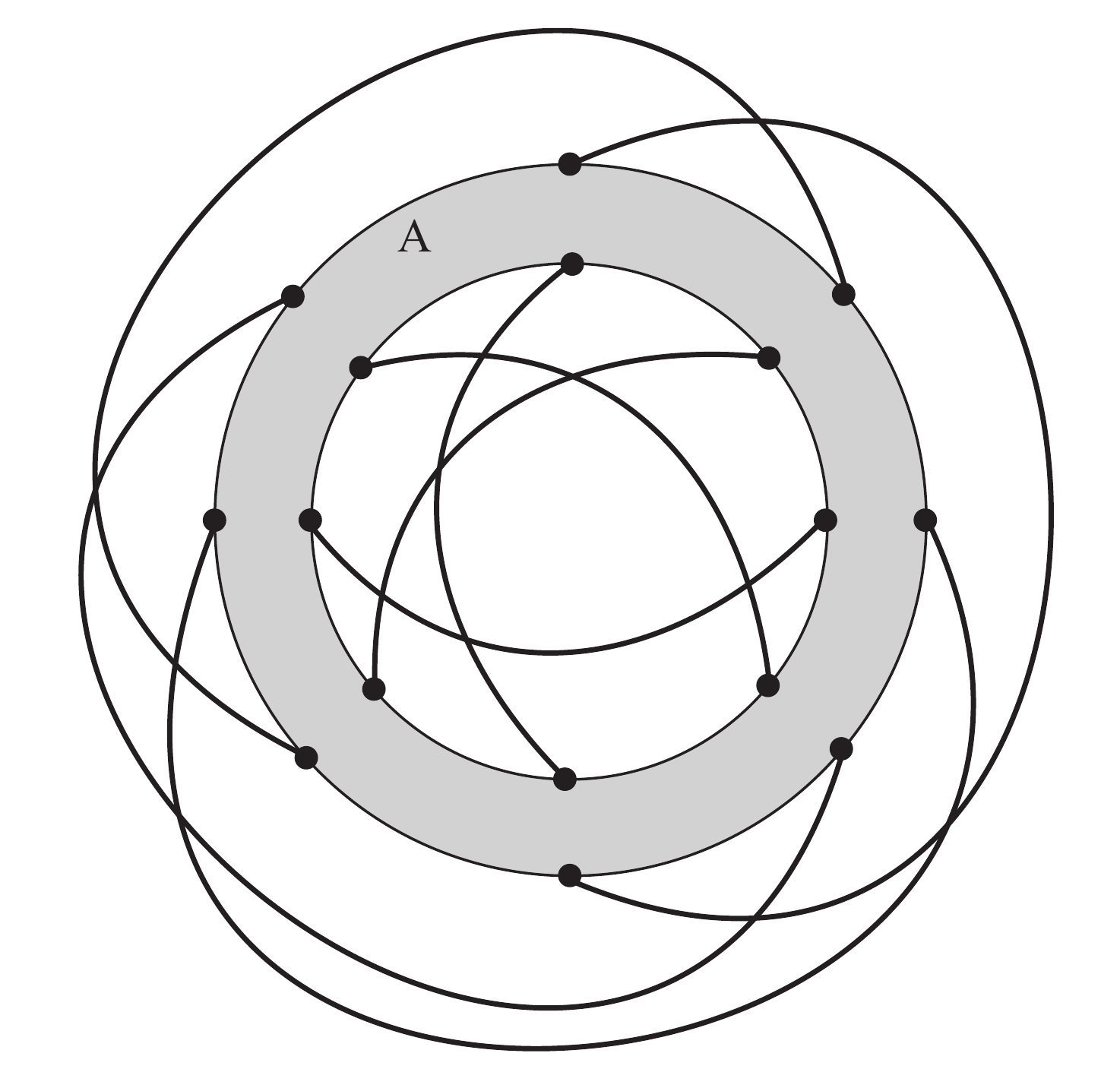}
    \caption{Closing off the two boundaries of $A$ with embroidery.}
    \label{closing off annulus}
\end{figure}

Using similar arguments as in the proof of Lemma \ref{spheres}, we check that $\overline{E''}$ satisfies the following conditions. By construction, $\overline{E''}$ is alternating and connected.  A similar argument as in the proof of Lemma \ref{spheres} shows $\overline{E''}$ to be reduced and $L'$ to be nontrivial. To see that $\overline{E''}$ is obviously prime, we need to consider some extra cases. Suppose again that there exists a disk $G \subset \mathbb{R}^2$ containing a crossing of $\overline{E''}$ such that $\partial G$ crosses $\overline{E''}$ transversely twice. We eliminate cases as before until the only case left is one in which both intersections on $\partial G$ are with embroidery arcs. Again as before, we eliminate the sub-case where both intersections on $\partial G$ are with the outer embroidery. By symmetry, we can also eliminate both intersections being with the inner embroidery.
If one intersection on $\partial G$ is with outer embroidery and the other is with the inner, then $G$ cannot contain an embroidery crossing, implying that $E''$ is not connected, a contradiction. 
%If both intersections on $\partial G$ are with the inner embroidery, consider again the complementary regions of the inner embroidery arcs on the plane. See Figure \ref{inner embroidery region}(a). $\partial G$ cannot cross any arc of an inner embroidery arc that is not incident to the region containing $A$, or $G$ can only contain a trivial arc. So consider only arcs of the embroidery arcs that are incident to the region containing $A$. They must have the pattern described in Figure \ref{inner embroidery region}(b). If $\partial G$ the previous argument goes through. 
To see that $\overline{E''}$ is not a 2-braid, we note that as in Figure \ref{closing off annulus}, the inner embroidery of $\overline{E''}$ contains regions that are not bigons, which cannot exist in a reduced alternating projection of a 2-braid.\\

%\begin{figure}
%    \centering
 %   \includegraphics[scale=.15]{inner embroidery region.png}
 %   \caption{Complementary regions of the inner embroidery arcs. The region containing $A$ is shaded and has the shape described in (b), where the ellipses denote the same number of 4-gon regions.}
%    \label{inner embroidery region}
%\end{figure}

Therefore, $\overline{E''}$ is a projection of a reduced non-trivial non-split prime alternating link that is not a 2-braid. Now we add the augmenting components to obtain $Q'$ again, and augment with one more component through the hole of the annulus, as shown in Figure \ref{donut hole augment}. We obtain the torus $T'$ in the complement of a new generalized augmented alternating link $Q''$ in a thickened annulus. This is a generalized augmented alternating link in $S^3$, which must therefore be hyperbolic. Hence, $T'$ must either be boundary parallel or compressible in $S^3 \setminus Q''$.\\

\begin{figure}
    \centering
    \includegraphics[scale=.3]{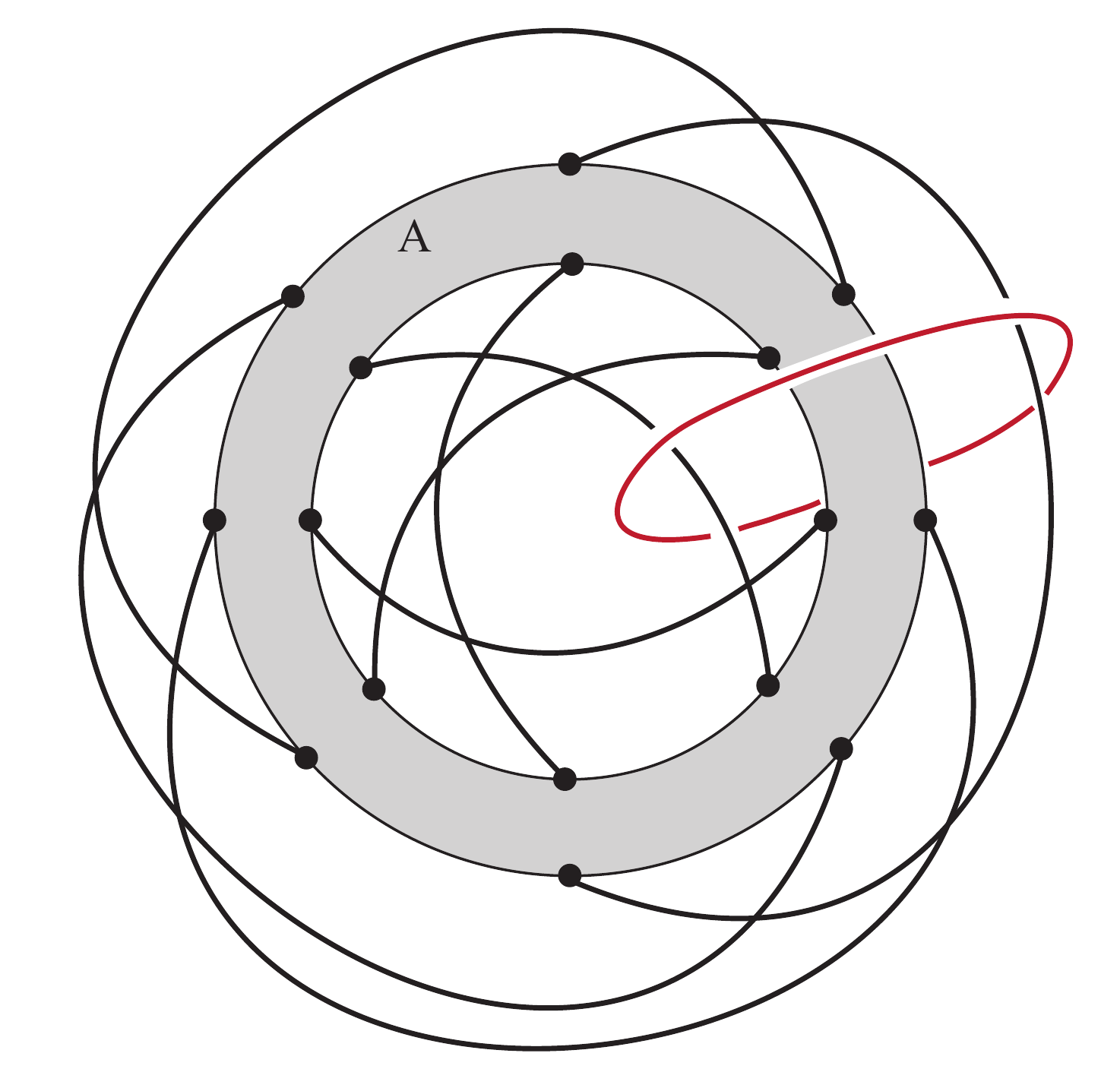}
    \caption{Augmenting once more to prevent compression of $T'$.}
    \label{donut hole augment}
\end{figure}

If $T'$ is boundary parallel, it must be boundary parallel to the side to which it bounds a solid torus and away from the boundary of the new link. Therefore $T'$ must have been boundary parallel in the universal cover, which implies $T$ must have been boundary parallel in the link complement $S\cross I\setminus Q$, a contradiction.\\

If $T'$ is compressible in $S^3 \setminus Q''$, it cannot compress outside the thickened annulus, since we removed the final augmenting component. So it must compress to obtain a sphere in the thickened annulus. Since $S^3 \setminus Q''$ is hyperbolic, the sphere must bound a ball to one side. So either all the components are to the inside or all to the outside. It cannot be the case that all components are inside, as the strand ends from the boundaries of $A$ are by default outside the sphere. If all the components are to the outside, the torus $T'$ compresses to the inside in $S\cross I\setminus Q$, so $T$ was not essential to begin with, a contradiction. 
%\textcolor{orange}{Do we need both this and the last augmentation?}
\end{proof}

\begin{lemma}\label{annuli}
There are no essential annuli in $M \setminus \mathring{N}(Q)$.
\end{lemma}

\begin{proof}
Let $A$ be an essential annulus in $S \times I \setminus Q$. As our annulus $A$ has two boundary components, exactly one of the following holds:

\begin{enumerate}
\item $A$ has boundary strictly on $\partial (S \times I)$,
\item $A$ has boundary strictly on $\partial(\mathring{N}(Q))$,
\item $A$ has boundary on both $\partial (S \times I)$ and $Q$.
\end{enumerate}

Note first that both boundaries of $A$ are nontrivial curves on $\partial (S \times I)$ or $\partial(\mathring{N}(Q))$. If both boundaries of $A$ are trivial, we can cap off $A$ with the disks bounded by the trivial boundaries to obtain an essential sphere, which we have eliminated in Lemma \ref{spheres}. If one boundary of $A$ is trivial and another boundary of $A$ is nontrivial, we can cap off the trivial boundary with the disk it bounds to obtain a disk in $S \times I \setminus Q$ bounded by the nontrivial boundary of $A$. However, no nontrivial curve on $\partial (S \times I \setminus Q)$ bounds a disk.\\

We split Case (1) into two sub-cases, characterized by whether the two components $\alpha$ and $\beta$ of $\partial A$ lie on exactly one or both boundary components of $S \times I$.\\

Suppose that the boundary curves $\alpha$ and $\beta$ lie on the same surface $S = S \times \{0\}$ or $S \times \{1\}$. Observe that $\alpha$ is isotopic through $A$ to $\beta$. This annulus remains in $S \times I \setminus L$, but cannot be essential in the original link complement by assumption of hyperbolicity. Suppose that $A$ is compressed by disk $D$ in $S \times I \setminus L$. Isotoping $\partial D$ to $\alpha$ implies the existence of a disk $D'$ by which $S$ may be compressed, which is impossible.\\

If $A$ is boundary parallel in $S \times I \setminus L$ but not $S \times I \setminus Q$, let $A'$ be the annulus on $S$ obtained by isotoping $A$ onto $S$ relative to $\alpha$ and $\beta$. We form a torus $T_A$ by gluing $A$ to $A'$ along their boundaries. Since $A$ is boundary parallel in $S \times I \setminus L$, then there must be only augmenting components $J_1, \ldots, J_n$ contained within $T_A$. Consider that all $J_i$ must be trivial in the solid torus bounded by $T_A$. To see this, the isotopy of $A$ onto $A'$ projects all $J_i$ to closed curves on $S$. Since the disks bounded by all $J_i$ are disjoint, then no two $J_i$ are linked, and so their projections can be disjoint on $S$ after an isotopy. If any $J_i$ were non-trivial in the solid torus bounded by $T_A$, then $J_i$ would also be non-trivial as a simple closed curve after projection onto $S$. This contradicts the assumption that $J_i$ bounds a disk in $S \times I$. Since all $J_i$ are trivial in the solid torus bounded by $T_A$, there exists a compressing meridian of $T_A$ which produces an essential sphere, a contradiction. \\

Suppose $A$ is an essential annulus in $S \times I \setminus Q$ with one boundary on the inner surface $S \times \{0\}$ and the other boundary on the outer surface $S \times \{1\}$. By the assumption that $S \times I \setminus L$ is hyperbolic, $A$ cannot be essential in $S \times I \setminus L$. As $A$ is not boundary parallel in $S \times I \setminus L$, then $A$ is compressible in $S \times I \setminus L$. For any compressing disk $D$, then $\partial D$ is isotopic through $A$ to a boundary of $A$ on $S \times \{1\}$. This implies the existence of a compressing disk $D'$ for $S \times \{1\}$, which is impossible. This concludes Case (1).\\

We now discuss Case (2) in three sub-cases, characterized by whether the two boundary curves of $A$ both lie on $\partial N(L)$, both lie on $\partial N(J)$, or one on $\partial N(L)$ and one on $\partial N(J)$.\\

%Suppose $A$ has one boundary $\alpha$ on $L_1$ and the other boundary $\beta$ on $L_2$, where $L_1$ and $L_2$ are distinct components of $L$. %In the case where $\alpha$ and $\beta$ are parallel, 
%Taking the boundary of a regular neighborhood of $A \cup L_1 \cup L_2$,  
%%In the case where $\alpha$ and $\beta$ are not parallel, by taking the boundary of a regular neighborhood of $A \cup L_1$, we obtain a new annulus $A'$ with both boundaries on $L_2$. We know $A'$ is not boundary parallel, and if there were a curve on $A'$ bounding a compressing disk of $A'$, an isotopic curve would bound a compressing disk of $A$. Thus $A'$ is essential. We reduce to the situation addressed in the next paragraph.\\

Suppose $A$ is an essential annulus in $S \cross I \setminus \mathring{N}(Q)$ with boundaries on distinct components $\partial N(L_1)$ and $\partial N(L_2)$ of $L$. The two boundary components are nontrivial curves on $\partial N(L_1)$ and $\partial N(L_2)$. Then $A$ is not essential in $S \cross I \setminus \mathring{N}(L)$. It is not boundary parallel because its boundaries are on distinct components, thus it must be compressible. There must exist a compressing disk $D$ with boundary a nontrivial curve on $A$. But then $D$ can be isotoped to have boundary on the boundary on $A$. We have thus constructed a compressing disk for a nontrivial curve on $\partial N(L_1)$ or $\partial N(L_2)$. This is a contradiction to the incompressibility of nontrivial  curves on the boundary of regular neighborhoods of the link components in $L$.\\

Suppose $A$ has both boundaries on $\partial N(C)$ for some component $C$ of $Q$, either from the link $L$ or an augmenting component.  The two boundaries of $A$ must be parallel curves. Note that the two  curves cannot be meridians on $\partial N(L)$ as we already proved that $Q$ is prime.\\

We can form a torus from $A$ and a second annulus $A'$ on $\partial N(C)$. Note that we have two choices for that second annulus on $\partial N(C)$, the torus generated by one of which, denoted $T_1$, separates $C$ from $\partial S \cross I$.  Since we have already eliminated essential tori, $T_1$ must either be boundary-parallel or compressible in $S \cross I \setminus Q$. If $T_1$ is boundary-parallel, it bounds a solid torus containing $C$ as core curve. But this implies that $A$ is also boundary-parallel,  a contradiction. \\

So $T_1$ must be compressible. But any compression yields a sphere that separates $C$ from $\partial S \cross I$. So it is an essential sphere, which is a contradiction. 

%If it compressed to the outside, it would generate a sphere with $L_1$ and possibly other components to the inside and $\partial S \cross I$ to the outside, so it would be an essential sphere, a contradiction.

%Hence, at least one must be compressible. But  then  $A$ is compressible as well, a contradiction. {\color{red} More to do here.}

%By a similar argument as in the case where the two boundaries of $A$ are parallel curves on two distinct components of $L$, we get a contradiction.{\color{green} I guess i don't see this immediately, so another sentence or two could be helpful. Then again, it might be totally fine if the readers are familiar with these types of argument}\\

Suppose $A$ has boundaries on two distinct augmenting components $J$ and $J'$. Then the boundary of a regular neighborhood of $J \cup J' \cup A$ is a torus $T_A$ that cannot be boundary-parallel since $J$ and $J'$ are to one side and the boundaries of $S \cross I$ are to the other side. If it compresses, the boundary of the compressing disk can be pushed onto $\partial N(J)$, a contradiction to the fact we have already eliminated any such disks. But then $T_A$ is essential, a contradiction to the fact we have eliminated any such tori.\\

%Take the universal cover. In the thickened Euclidean or hyperbolic plane, since augmenting components are trivial in $S \times I$, $J$ and $J'$ lift to two circles. Thus $A$ lifts to a new annulus bounded by the two circles. A similar argument as in the the proof of Lemma \ref{spheres} gives a contradiction.\\

%Suppose $A$ has boundaries on the same augmenting component $J$. {\color{red} Taken care of above.}Then, there exists another augmenting component $J'$ between $A$ and $\partial N(J)$, which prevents $A$ from being boundary parallel in $S \times I \setminus Q$. Since $J$ and $J'$ are nonisotopic, $J'$ bounds a disk in the solid torus bounded by $A$ and $A' \subset \partial N(J)$, the annulus isotopic to $A$ relative to its boundary. 
%{\color{red} Check this.}

%By a similar argument as in the case where both boundaries of $A$ are on the same boundary of $S \times I$, we get a contradiction. \textcolor{blue}{I think this should be fleshed out more/ sketched in a figure}\\

Suppose $A$ has one boundary on $\partial N(L)$ and the other boundary on $\partial N(J)$. Taking the boundary of a regular neighborhood of $A \cup L$ (or $A \cup J$), we obtain a new annulus $A'$ with both boundaries on $J$ (or the same component of $L$). The annulus $A'$ is essential, since it is not boundary parallel, and if there were a curve on $A'$ bounding a compressing disk of $A'$, an isotopic curve would bound a compressing disk of $A$, which is impossible. But we have just eliminated essential annuli with both boundaries on $J$ (or the same component of $L$). This completes Case (2).\\

For Case (3), suppose $A$ has one boundary on $\partial N(Q)$ and the other boundary on $S \times \{1\}$ or $S \times \{0\}$. Taking the boundary of a regular neighborhood of $A$ and the component of $Q$ on which $\partial A$ lies, we obtain a new annulus $A'$ with both boundaries on the same component of $\partial(S \cross I)$. the annulus $A'$ is essential, since it is not boundary parallel by construction, and if there were a curve on $A'$ bounding a compressing disk of $A'$, an isotopic curve would bound a compressing disk of $A$, which cannot exist. But we have already eliminated essential annuli with both boundaries on the same component of $\partial(S \cross I)$ in Case (1).
\end{proof}

\begin{proof}[Proof of Theorem \ref{orientable theorem}]
By Lemmas \ref{spheres}, \ref{tori} and \ref{annuli} there are no essential spheres, tori or annuli in $M \setminus Q$.  Therefore, a generalized augmented cellular alternating link in a thickened orientable surface is hyperbolic.
\end{proof}

\section{Generalized Augmented cellular Alternating Links in $I$-Bundles}\label{non-orientable section}

We now extend to more general $I$-bundles over orientable or non-orientable surfaces. 

\begin{definition}
Let $S$ be a closed surface, orientable or not and let $N$ be an $I$-bundle over $S$. We exclude $N$ being $S \cross I$ when $S$ is the projective plane. Otherwise, define $M$ to be $N$ when the Euler characteristic of $S$ is negative. Define $M$ to be $N \setminus \partial N$ when $S$ has Euler characteristic 0, and define $M$ to be $N$ with its spherical boundaries capped off with balls in the case the Euler characteristic of $S$ is positive.
\end{definition} 

Let $L$ be a prime link in $N$ with a reduced cellular alternating projection to $S$. 
Note that a projection on a surface that is alternating on the surface when considered in $S \cross I$ may not be alternating on the surface when considered in a twisted $I$-bundle over the surface.

\begin{definition}\label{main definition}
Let $J_1, ..., J_n$ be $n$ loops in $N \setminus L$ which are trivial in $N$ such that

\begin{enumerate}
    \item each $J_i$ intersects $S$ in exactly two points, one in each of two non-adjacent complementary regions of $L$,
    \item If two $J_i$ intersect $S$ in the same pair of regions,  they are not isotopic in $N \setminus L$. 
    \item each $J_i$ bounds a disk $E_i$ perpendicular to $S$ in $N$ where $E_i \cap E_j = \varnothing$ for $i \neq j$.
\end{enumerate}

Then $Q = L \cup (J_1 \cup \dots \cup J_n)$ is said to be a {\it generalized augmented cellular alternating link} in $N$.

\end{definition}

Let us recall Theorem \ref{general theorem}:

\begingroup
\def\thetheorem{\ref{general theorem}}
\begin{theorem}
Let $Q$ be a generalized augmented cellular alternating link in $N$, an $I$-bundle over a closed surface $S$, excluding $S \cross I$ when $S$ is a projective plane. Then $M \setminus Q$ has a complete hyperbolic metric, unless one of the following occurs:
\begin{enumerate} 
\item $S$ is a sphere and $L$ is a 2-braid knot or link.
\item $S$ is a projective plane and $L$ is a 2-braid link.
\item $S$ is a projective plane and there exists a simple closed curve on $S$ that intersects the projection once.
\end{enumerate}
  If $\chi(S) <  0$, then $\partial M$ can be realized as totally geodesic surfaces in the hyperbolic metric on $M \setminus Q$.
\end{theorem}
\addtocounter{theorem}{-1}
\endgroup

%\begin{theorem}\label{}

\begin{proof}
%{\color{magenta} I think this lemma only applies in the case that the Euler characteristic is negative, i.e. not the projective plane or klein bottle.}{\color{red} I think it is okay except for the exceptions listed and proj. plane $\cross I$ , which has been previously removed from consideration.}

We have already proved the result for $N$ orientable and $S$ orientable. Utilizing results from the proof of Lemma 14 in \cite{SMALL2017}, we consider four cases:

\begin{enumerate}

\item  $N = S \tilde{\times} I$, where $N$ is non-orientable and $S$ is orientable. Then, $S$ cannot be a sphere, and $N$ is double-covered by $\tilde{N} = \tilde{S} \times I$ for $\tilde{S}$ an orientable double cover of $S$.

\item $N = S \times I$, where $N$ is non-orientable and $S$ is non-orientable. Then, $N$ is double-covered by $\tilde{N} = \tilde{S} \times I$ for $\tilde{S}$ an orientable double cover of $S$.

\item $N = S \tilde{\times} I$, where $N$ is orientable and $S$ is non-orientable. Then, $N$ is double-covered by $\tilde{N} = \tilde{S} \times I$ for $\tilde{S}$ an orientable double cover of $S$.

\item $N = S \tilde{\times} I$, where $N$ is non-orientable and $S$ is non-orientable. Then, $N$ is double-covered by $\tilde{N} = \tilde{S} \tilde{\times} I$ for $\tilde{S}$ non-orientable and $\tilde{N}$ orientable.
\end{enumerate}

 In \cite{SMALL2017}, it is shown that in all four cases the lift $\tilde{L}$ of the link $L$ to $\tilde{N}$ is prime and cellular alternating in $\tilde{N}$.
It remains to check that $\tilde{L}$ is reduced. Following \cite{SMALL2017}, we think of $S$ as a polygon $R$ with edges identified. Then $\tilde{S}$ is two copies of $R$ glued together with edges identified, denoted by $\tilde{R}$. Suppose there is a Type I crossing in the projection of $\tilde{L}$ onto $\tilde{R}$. $L$ is reduced, so to avoid having a Type I crossing in $R$, it must be the case that the crossing occurs where the two copies of $R$ are glued. But since identified edges of $R$ are glued, in a single copy of $R$, there is also a Type I crossing, a contradiction.\\

Now, we consider the augmenting components. Given an augmenting component $J_i$ in $N$, since it is trivial in $N$, it lifts to two  trivial components $\tilde{J_i}_1$ and $\tilde{J_i}_2$ in $\tilde{N}$. Each $\tilde{J_i}_j$ intersects $S$ in exactly two points, and since these two points must occur in a single copy of $R$, they must occur one in each of two non-adjacent regions.  Finally, note that each $\tilde{J_i}$ bounds a perpendicular disk in $\tilde{N}$, and again since $J_i$'s do not intersect $\partial N$, their lifts bound disjoint disks in $\tilde{N}$.\\

Therefore, in Cases (1), (2), and (3), the lift of $Q$, $\tilde{Q}$, is a generalized augmented cellular alternating link in a thickened orientable surface, and by Theorem \ref{orientable theorem}, $\tilde{Q}$ is hyperbolic in the manifold $M'$ corresponding to capping off sphere boundaries or shaving off torus boundaries. 
By the Mostow Rigidity Theorem, we know that the deck transformation associated to the covering $\tilde{N} \rightarrow N$ can be realized as an isometry of $M'$. The quotient under the action of this isometry, which is $M \setminus Q$, is then hyperbolic.\\
%{\color{magenta} The Mostow Rigidity proof is a spook.
 %   This does, however, follow quickly from Geometrization;
 %   see Lemma 1.1 in lowerboundsvolumes.}{\color{red} ? If the double cover is hyperbolic, Mostow Rigidity does say the orientation-reversing covering transformation is an isometry and the quotient is then hyperbolic. We do not need Geometrization here.}
 In Case (4), $\tilde{Q}$ is a generalized augmented cellular alternating link in an orientable twisted $I$-bundle over a non-orientable surface. This case then reduces to Case (2).
\end{proof}

\section{Rubber Band Links}\label{rubber band section}
\begin{definition}
Let $G$ be a graph on a closed orientable surface $S$. Call $G$ a \textit{cage graph} if it  is connected, has no parallel edges, all complementary regions are disks, and no trivial circle in $S$ intersects $G$ exactly once. 
\end{definition}
\begin{definition}
Let $G$ be a cage graph on a closed orientable surface $S$. Replace each vertex $v_i$ with an unknotted component $V_i$, called a vertex component, lying in $S$ that is the boundary of a small disk neighborhood of $v_i$. Replace each edge $e_i$ with an unknotted component $E_i$, called an edge component, that projects into the edge and such that it intersects $S \times \{\frac{1}{2}\}$ in exactly two points, where each of the two intersections occurs in the two different regions bounded by the two trivial components surrounding the edge's two vertices. We call the link made up of $V_1, \ldots, V_m,  E_1, \ldots, E_n$ a \textit{rubber band link} in $S \times I$.
\end{definition}

\begin{figure}[h]
    \centering
    \includegraphics[scale=.7]{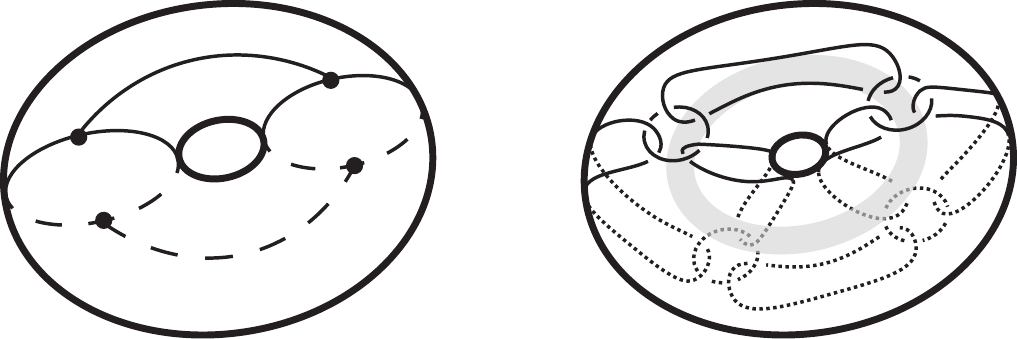}
    \caption{A rubber band link.}
\end{figure}

In the case that a rubber band link is hyperbolic, it has many very nice properties. For instance, all of the twice-punctured disks bounded by the edge components are totally geodesic. Moreover, there is a reflection in the projection surface that preserves the thickened surface and the link. Hence, its fixed point set is totally geodesic as well. In particular, this means that the punctured disks bounded by the vertex components are all totally geodesic. In addition, the punctured surface obtained from the projection surface when we cut along the vertex components is also totally geodesic. The twice-punctured disks bounded by the edge components must intersect all of the totally geodesic surfaces on the projection surface in right angles. The link complement can be decomposed into two pieces, each of which is topologically a thickened surface but each of which has an ideal polygonal boundary on one side with all neighboring polygons meeting at right angles, as in the decompositions of fully augmented alternating links in the 3-sphere \cite{Purcell} and $T \times I$ \cite{Kwon}.

\begin{corollary} \label{rubber band links}
Every rubber band link has hyperbolic complement in $S \times I$.
\end{corollary}

\begin{proof}
Let $R$ be a rubber band link. Note that each edge component $E_i$ of $R$ bounds a twice-punctured disk, or equivalently a thrice-punctured sphere $F_i$. We utilize the results of \cite{Adams3}, where it is proven that cutting open an oriented finite volume hyperbolic 3-manifold along an incompressible thrice-punctured sphere $F$ and then reidentifying the two copies of $F$ by any orientation-preserving homeomorphism of $F$ will give another hyperbolic 3-manifold with the same volume. At each $F_i$, we perform an operation described in Corollary 5.1 of \cite{Adams3}, but for $R$ in $S \times I$: Cut $S \times I \setminus R$ open along $F_1$, twist a half twist, and re-identify. Figure \ref{half twist} and Figure \ref{half twist graph} provide illustration. Repeat the operation for $F_2, \dots, F_n$. We obtain a new link complement $S \times I \setminus Q$. Note that $E_i$ are not changed in the process, but we will replace notation for each $E_i$ in $R$ by $E_i'$ in $Q$.\\

\begin{figure}[h]
    \centering
    \includegraphics[scale=.3]{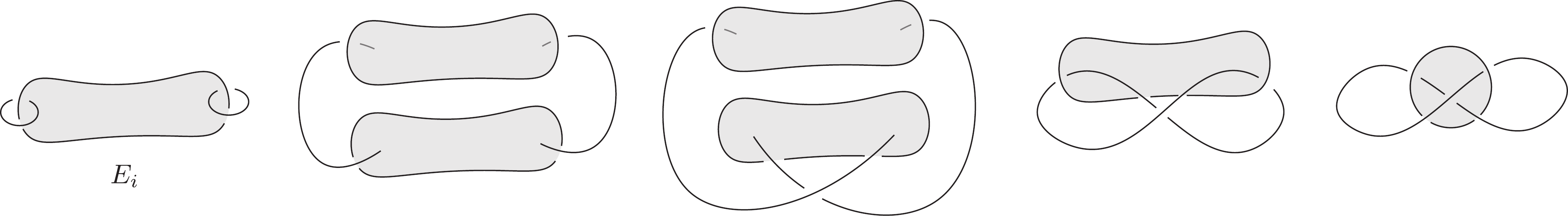}
    \caption{Performing a half-twist operation at an $F_i$.}
    \label{half twist}
\end{figure}

\begin{figure}[h]
    \centering
    \includegraphics[scale=.7]{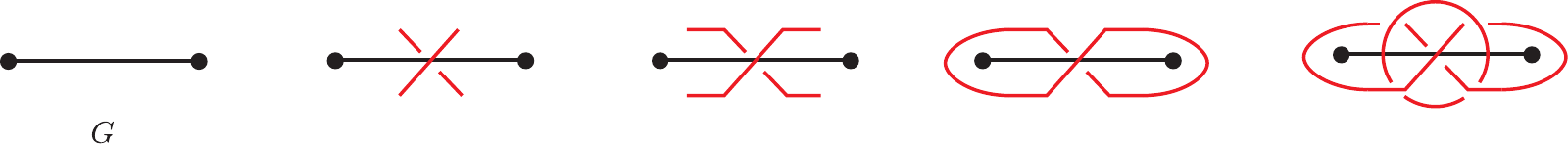}
    \caption{An equivalent way of obtaining the result of the operation directly from the original graph.}
    \label{half twist graph}
\end{figure}

We can choose the half twists so that $Q$ is a generalized augmented cellular alternating link. To see this, consider $L = Q \setminus (E_1 \cup \cdots \cup E_n)$. A projection $P$ of $L$ to $S \times \{\frac{1}{2}\}$ can be obtained in the following way: At each edge of $G$, place a crossing so that the edge bisects the crossing. Now consider each vertex $v$ of $G$. For the crossings on the edges it is incident to, extend the two strands closer to $v$ of each of these crossings until these strands hit a circle of radius $\epsilon$ about $v$. Let the set of arcs in this circle whose endpoints are not on strands from the same crossing be in the projection. See Figure \ref{crossing at edge}.\\

\begin{figure}[h]
    \centering
    \includegraphics[scale=.5]{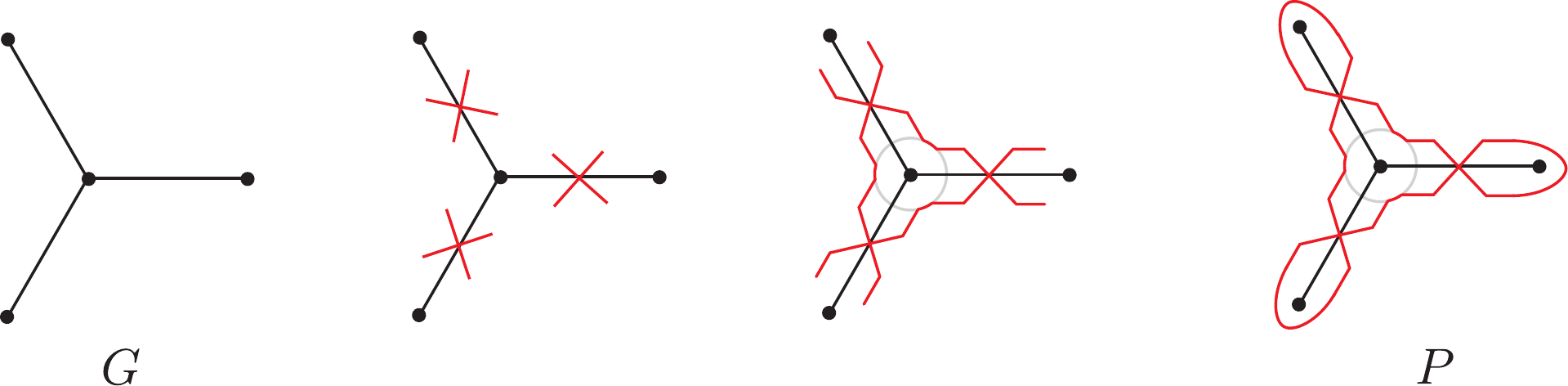}
    \caption{Obtaining a projection $P$ of $L$ directly from $G$.}
    \label{crossing at edge}
\end{figure}

$P$ is reduced: Suppose it has a Type I crossing. Then there exists a vertex incident to 1 edge, i.e., a leaf. But then there exists a trivial circle in $S$ intersecting $P$ exactly once. Suppose it has a crossing that can be removed by a flype. Then there exists a sphere in $S \times I$ punctured exactly twice by the two strands of $L$ near this crossing. In the projection onto $S \times \{\frac{1}{2}\}$, this sphere projects to a trivial circle intersecting $P$ exactly twice by the two strands of $P$ near this crossing. Then this circle intersects $G$ exactly once, a contradiction to our condition.\\

$L$ is cellular alternating: Each crossing corresponds to a half twist, and we can always choose our half twists so that $L$ is alternating. Also, all complementary regions of $L$ are disks: Note first that all complementary regions of $L$ containing a vertex are disks. Note also that all other complementary regions of $L$ are isotopic to the complementary regions of $G$, and so are disks.\\

Since $L$ is cellular alternating, we know $P$ is a connected projection of $L$: Suppose not. Then there exists a  simple closed curve $\alpha$ on $S$ avoiding $P$ that does not bound a disk on $S$ disjoint from $P$. But $\alpha$ must be contained in a complementary region, and as every complementary region of $L$ is a disk, $\alpha$ must bound a disk on $S$ disjoint from $P$, a contradiction.\\

$L$ is prime: Since $P$ is a reduced cellular alternating projection, by \cite{SMALL2017}, $L$ is prime if and only if $P$ is obviously prime. Suppose $L$ is not prime, i.e., there exists a disk $D$ in $S$, whose boundary intersects $P$ transversely at two points, containing a crossing of $P$. As this crossings of $P$ correspond to an edge $e$ of $G$, and the strands of the crossing follow along $e$, $e$ cannot intersect $\partial D$ twice, since otherwise $\partial D$ intersects $P$ at least 4 times. By our condition for $G$, $e$ cannot intersect $\partial D$ exactly once. If $e$ does not intersect $\partial D$, there must exist a path $e, \ldots, e'$ in $G$ such that $e'$ intersects $\partial D$ exactly once, for otherwise $\partial D$ would not intersect $P$ twice. This again contradicts our condition for $G$. Therefore, $L$ is a reduced prime cellular alternating link in $S \times I$.\\

Now we consider components $E_1', \ldots, E_n'$. Note first that each $E_i'$ corresponding to edge $e_i$ intersects $S \times \{\frac{1}{2}\}$ in exactly two points, one in each of the two regions corresponding to the two vertices incident to $e_i$, and these two regions are not adjacent. Suppose $E_i'$ and $E_j'$ intersect $S \times \{\frac{1}{2}\}$ in the same pair of regions. Then their corresponding edges, $e_i$ and $e_j$, are parallel edges, a contradiction to our condition for $G$. Finally, note that each $E_i'$ bounds a perpendicular disk in $S \times I$, and all such perpendicular disks are disjoint.\\

Therefore, $Q$ is a generalized augmented cellular alternating link. Then, there exists a sequence of operations, where we repeatedly cut $S \times I \setminus Q$ along an $F_i'$, the twice-punctured disk bounded by $E_i'$, twist a half twist, and reidentify, which yields $R$. In the special case that $L$ is a 2-braid and $S$ is a sphere, we can replace a half twist with a full twist for one of the $E_i$, adding a crossing to $P$, and changing it from being a 2-braid.\\

Note that at each intermediate stage in the process, the $F_i'$ we are about to operate on is incompressible in $S \times I \setminus Q$. To see this, suppose there did exist a compressing disk $D$ of $F_i'$. If $\partial D$ loops exactly one puncture on $F_i'$, then $D$, together with the once-punctured disk on $F_i'$ bounded by $\partial D$, forms a once-punctured sphere in $S \times I \setminus Q$, which is impossible. If $\partial D$ loops both punctures on $F_i'$, then it is isotopic through $F_i'$ to $E_i'$. Then $E_i'$ bounds a disk in $S \times I \setminus Q$. Taking a neighborhood of this disk, we obtain an essential sphere in $S \times I \setminus Q$, with $E_i'$ to one side of the sphere. This contradicts hyperbolicity of $Q$ as stated by Theorem \ref{orientable theorem}.\\

Therefore, as $M \setminus Q$ is hyperbolic by Theorem \ref{orientable theorem},  Theorem 4.1 of \cite{Adams3} implies $M \setminus R$ is hyperbolic with the same volume.
\end{proof}

\section{Further Applications}\label{applications section}

A twist region in a link projection is a sequence of end-to-end bigons, including a single crossing that does not touch a bigon as a twist region. A projection is twist-reduced if flypes have been applied to minimize the number of distinct twist regions. \\

In \cite{Lackenby}, both lower and upper bounds on the volume of a prime alternating link in $S^3$ were provided in terms of the number of twist regions present in a twist-reduced alternating diagram of the link. Corollaries $2$ and $3$ of the same paper apply this lower bound to prove that hyperbolic alternating links and augmented alternating links in $S^3$ together form a closed subset of the set of hyperbolic links in $S^3$ in the geometric topology.\\

Recently, J. Howie and J. Purcell proved a lower volume bound for \textit{twist-reduced} weakly alternating links in thickened surfaces as a function of twist number in \cite{HP}. This lower bound is restated in \cite{KP} with an upper bound that will be used later, so we will refer to both bounds as they appear in \cite{KP}.\\

Arguments following almost exactly as those ones made in \cite{Lackenby} can be made for augmented cellular alternating links in thickened surfaces by applying the lower bound in \cite{HP}. In this case, we restrict  augmentations from the general case, so that the disks bounded by augmented components are punctured exactly twice by the underlying cellular alternating link.
However, we first verify that the class of links considered throughout this paper indeed satisfy the requirements of \cite{KP}.

\begin{lemma} \label{twist-reduced WGA}
If $K$ is a hyperbolic cellular alternating link in $S \times I$, then there is a twist-reduced weakly generalized alternating projection $\pi(K)$ on $S \times \{\frac{1}{2}\}$.
\end{lemma}
\begin{proof}
If $K$ is a hyperbolic cellular alternating link in $S \times I$, then a cellular alternating diagram $\pi(K)$ is weakly generalized alternating (WGA) as in \cite{KP}, since it is reduced alternating, admits a checkerboard coloring, and trivially satisfies the representativity requirement of the WGA definition. Namely, the projection surface $S \times \{\frac{1}{2}\}$ is incompressible. So cellular alternating links admit WGA diagrams, and by \cite{HP}, any WGA diagram can be made to be twist reduced through flypes, so that the resulting diagram is twist reduced and still WGA.
\end{proof}
\begin{theorem} \label{closure in geometric topology}
Hyperbolic cellular alternating links and augmented cellular alternating links form the closure of hyperbolic cellular alternating links in $S \times I$ under the geometric topology.
\end{theorem}
\begin{proof}

 As is shown in \cite{Lackenby} for the case of $S^3$, the cellular alternating link $K$ can be obtained by Dehn surgery on an augmented cellular alternating link $L$ with diagram $\pi(L)$ having at most $6t(\pi(K))$ crossings, where $t(\pi(K))$ is the twist number of the projection. 
 By Lemma \ref{twist-reduced WGA}, if $K$ is a hyperbolic cellular alternating link in $S \times I$, then there is a twist-reduced weakly generalized alternating projection $\pi(K)$ on $S \times \{\frac{1}{2}\}$. It then follows from Theorem $1.4$ of \cite{KP} that the twist number of $\pi(K)$ is bounded above by $\text{vol}(M \setminus K)$, so the number of crossings in $L$ is bounded. Afterwards, the proof follows exactly as those of Corollary 2 and Corollary 3 of \cite{Lackenby} by the fact that convergence of complete finite volume hyperbolic $3$-manifolds to a limit finite volume hyperbolic 3-manifold $M_\infty$ is characterized by Dehn surgery on $M_\infty$.
\end{proof}

Recall as constructed in Corollary \ref{rubber band links} that if $G$ is a cage graph on a closed orientable surface $S$, then the rubber band link $R = V_1  \cup \cdots V_m \cup E_1 \cup \cdots \cup E_n$ is hyperbolic with the same volume as the fully augmented cellular alternating link $Q = L \cup E_1' \cup \cdots \cup E_n'$, for a cellular alternating link $L$ in $M$. The link $L$ is hyperbolic in $M$ by \cite{SMALL2017}, and Dehn filling each augmenting component of $Q$ produces $L$. So, $vol(M \setminus L) < vol(M \setminus Q)$ (for instance, see Theorem $E.7.2$ of \cite{BP}), and the lower volume bounds on alternating links from \cite{Lackenby} and \cite{HP} apply to rubber band links in the thickened surface $S \times I$. In the special case $\chi(S) \ge 0$, then lower bounds from \cite{Purcell} and 
\cite{Kwon} exploit the fact the augmented alternating link $Q$ is fully augmented. \\

Following the lower bound developed in \cite{HP}, an upper bound on volume for twist-reduced weakly alternating links in the thickened surface was proved by E. Kalfagianni and J. Purcell in \cite{KP}. As before, these bounds do not apply directly to a rubber band link $R$ or to the fully augmented cellular alternating link $Q$ derived from the rubber band link. However, we can apply the upper bound to a sequence of cellular alternating links $\{L_i\}$ which approach $Q$ in the geometric topology by increasing the number of bigons in each twist region of $L$, corresponding to Dehn surgeries of increasing slope on augmenting components. These results are summarized in the following theorem relating the graph information of a rubber band link to its hyperbolic volume. There are also upper volume bounds from \cite{Lackenby}, \cite{KP} and \cite{Kwon} which apply to $L$, which we can show similarly to be upper volume bounds on $R$.\\

\begin{theorem}
Let $G = (V,E)$ be cage graph on the closed orientable surface $S$, and let $R$ be the rubber band link in $S \times I$ obtained from $G$. Let $\varepsilon = |E|$ denote the number of edges in $G$. 
\begin{enumerate}
\item If $\chi(S) = 2$, then 
\begin{align*}
   2(\varepsilon-1)v_\text{oct} \le \text{vol}(M \setminus R) \le 10(\varepsilon-1)v_{\text{tet}}
\end{align*}
    \item If $\chi(S) = 0$, then 
    \begin{align*}
       2\varepsilon v_{\text{oct}} \le \text{vol}(M \setminus R) \le 10  \varepsilon v_{\text{tet}}
    \end{align*}
    
    \item If $\chi(S) <0$, then
    \begin{align*}
      \frac{v_{oct}}{2}(\varepsilon - 3 \chi(S)) <  \text{vol}(M \setminus R) \le 6 \varepsilon v_{oct}
    \end{align*}
 
    \end{enumerate}

where $v_{oct} = 3.6638\dots$ is the volume of a regular ideal octahedron and $v_{tet} = 1.0149\dots$ is
the volume of a regular ideal tetrahedron.
\end{theorem}

\begin{proof}
If $\chi(S) = 2$, then $G$ is a cage graph on the $2$-sphere. There is a stereographic projection of $S$ which takes $G$ to a planar graph in $S^3$. Then, by the same operations as in the proof of Corollary \ref{rubber band links}, we obtain an augmented cellular alternating link $Q$ in $S^3$. In fact, the link $Q$ is fully augmented, since $Q$ is composed of an underlying alternating link $L$ together with augmenting components $E_1', \cdots, E_m'$, so that each crossing of $L$ is surrounded by exactly one $E_i'$. Then, the lower volume bound on $S^3 \setminus Q$ follows by Proposition $3.6$ of \cite{Purcell} (introduced in \cite{futer2006dehn}). \\

The upper bound proven by I. Agol and D. Thurston in the appendix to \cite{Lackenby} applies only to the underlying alternating link $L$ of $Q$. We can show that this bound applies to $Q$ as well by Dehn surgery. That is, by Theorem $E.5.1$ of \cite{BP} there exists a sequence of manifolds $\{S \times I \setminus L_i \}$ obtained by hyperbolic Dehn filling the augmenting components $E_1', \cdots, E_n'$ of $Q$ in $S \times I$, such that $vol(S \times I \setminus L_i) \rightarrow vol(S \times I \setminus Q)$. Moreover, each $L_i$ is alternating and twist-reduced, since the Dehn fillings only increase the number of bigons in each twist region of $L$. Since the number of twist regions in $L$ is $\varepsilon$, then $vol(S \times I \setminus L_i ) \le 10(\varepsilon-1)v_{tet}$, and so $vol(S \times I \setminus Q) \le 10(\epsilon-1)v_{tet}$.\\

If $\chi(S) = 0$, the upper and lower bounds follow from Proposition 3.5 of \cite{Kwon} if we show that the link $Q$ obtained from $R$ as in Corollary \ref{rubber band links} is a fully augmented link with $|E|$ crossing circles. By the proof of this corollary, $Q$ is fully augmented cellular alternating, with the link components $E_1, \cdots, E_n$ of $R$ becoming the augmenting components $E_1', \cdots, E_n'$ of $Q$. Hence, the number of crossing circles $c$ in Proposition 3.5 of \cite{Kwon} is $|E|$. \\ 

We remark that when $\chi(S) = 0$, the lower bound on $vol(S \times I \setminus R)$ from \cite{Kwon} is stronger than the lower bound in \cite{HP}, because the former applies directly to $Q$, while the latter \cite{HP} can only be applied to the cellular alternating sublink $L$ of $Q$. The upper bounds given by \cite{Kwon} and \cite{HP} coincide. \\

If $\chi(S) < 0$, it follows by Lemma \ref{twist-reduced WGA} that $L$ admits a projection $\pi(L)$ that is twist reduced and WGA, since $L$ is cellular alternating by the proof of Corollary \ref{rubber band links}. Moreover, $t(\pi(L)) = \varepsilon$.\\

For the lower bound, consider that any hyperbolic link $L'$ obtained by Dehn filling the augmenting components $E_1', \cdots, E_n'$ of $Q$ has the same number of twist regions as $L$; moreover, $L'$ is alternating if $L$ is. As noted previously, a hyperbolic manifold $N$ obtained by Dehn filling another hyperbolic manifold $M$ satisfies $vol(M) > vol(N)$. Let $L'$ be a hyperbolic link in $S \times I$ obtained from Dehn filling the augmenting components of $Q$. Then,
\begin{align*}
    \frac{v_{oct}}{2}(\varepsilon - 3\chi(S) < vol(S \times I \setminus L') < vol(S \times I \setminus Q)
\end{align*}

For the upper bound, we again use Theorem $E.5.1$ of \cite{BP} to obtain a sequence of hyperbolic links $\{L_i\}$ in $S \times I$ such that $vol(S \times I \setminus L_i) \rightarrow vol(S \times I \setminus Q)$. Since each $vol(S \times I \setminus L_i) \le 10(\varepsilon-1)v_{tet}$, then $vol(S \times I \setminus Q) \le 10(\varepsilon-1)v_{tet}$.\\
\end{proof}

\bibliography{references}
\bibliographystyle{plain}

\end{document}